\documentclass{amsart}
\usepackage{amssymb}
\usepackage[all]{xy}
\newcommand\datver[1]{\def\datverp%
 {\par\boxed{\boxed{\text{#1; Run: \today}}}}}
\newcommand\boxb[1]{\square_b}

\newcommand\ff{\operatorname{ff}}

\numberwithin{equation}{section}

\newcommand\paperbody%
        {
        
        }
\renewcommand{\geq}{\geqslant} \renewcommand{\leq}{\leqslant}
\newtheorem{lemma}{Lemma}
\newtheorem{proposition}[lemma]{Proposition}
\newtheorem{corollary}[lemma]{Corollary}
\newtheorem{theorem}[lemma]{Theorem}
\newtheorem{non-theorem}{Non-Theorem}

\theoremstyle{remark}
\newtheorem{definition}[lemma]{Definition}
\newtheorem{remark}[lemma]{Remark}

\numberwithin{lemma}{section}

\newcommand\Sm{\operatorname{Sm}}
\newcommand\Bi{\operatorname{Bi}}
\newcommand\Nc{\operatorname{Nc}}

\newcommand\ver{\mathrm{v}}
\newcommand\nver{\mathrm{nv}}

\newcommand\scat{\operatorname{sc}}

\newcommand\Vb{\mathcal{V}_{\operatorname{b}}}
\newcommand\bV{\mathcal{V}_{\operatorname{b}}}

\newcommand\bo{\operatorname{b}}

\newcommand\tot{\operatorname{tot}}

\newcommand\bT{{}^{\bo}T}

\newcommand\coF{{}^{\mathcal{C}}\kern-2pt\Lambda}

\newcommand\cFTs{{}^{\Phi}\overline{T}\kern-1pt{}^*}

\newcommand\tB{\tilde{B}}
\newcommand\tI{\tilde{I}}
\newcommand\tJ{\tilde{J}}

\newcommand\fa{\mathfrak{a}}
\newcommand\fb{\mathfrak{b}}
\newcommand\fc{\mathfrak{c}}

\newcommand\cB{\mathcal{B}}
\newcommand\cD{\mathcal{D}}
\newcommand\cE{\mathcal{E}}

\newcommand\cH{\mathcal{H}}
\newcommand\caH{\mathcal{H}}
\newcommand\cG{\mathcal{G}}

\newcommand\cM{\mathcal{M}}
\newcommand\cP{\mathcal{P}}

\newcommand\bbN{\mathbb N}

\newcommand\bbR{\mathbb R}
\newcommand\bbS{\mathbb S}

\newcommand\cC{\mathcal C}

\newcommand\CI{{\mathcal{C}}^{\infty}}

\newcommand\Diag{\operatorname{Diag}}

\newcommand\cFNs{{}^{\Phi}\overline N\kern-1pt{}^*}

\newcommand\codim{\operatorname{codim}}

\newcommand\pa{\partial}

\newcommand\Mand{\text{ and }}

\newcommand\Meither{\text{ either }}

\newcommand\Mif{\text{ if }}

\newcommand\Min{\text{ in }}

\newcommand\Mor{\text{ or }}
\newcommand\Morif{\text{ or if }}

\newcommand\Mst{\text{ s.t. }}

\newcommand\Mthen{\text{ then }}

\newcommand\Mwith{\text{ with }}
\newcommand\ob{\operatorname{ob}}

\newcommand\ncnt{n\@.c\@.n\@.t\@.}

\datver{1.0B; Revised: 14-8-2008}
\begin{document}
\title{Scattering configuration spaces}

\author{Richard Melrose}
\address{Department of Mathematics, Massachusetts Institute of Technology}
\email{rbm@math.mit.edu}
\author{Michael Singer}
\address{School of Mathematics, University of Edinburgh}
\email{M.Singer@ed.ac.uk}
\dedicatory{\datverp}
\begin{abstract} For a compact manifold with boundary $X$ we introduce the
  $n$-fold scattering stretched product $X^n_{\scat}$ which is a compact
  manifold with corners for each $n,$ coinciding with the previously known
  cases for $n=2,3.$ It is constructed by iterated blow up of boundary faces and
  boundary faces of multi-diagonals in $X^n.$ The resulting space is shown to map
  smoothly, by a b-fibration, covering the usual projection, to the lower
  stretched products. It is anticipated that this manifold with corners, or
  at least its combinatorial structure, is a universal model for phenomena
  on asymptotically flat manifolds in which particle clusters emerge at
  infinity. In particular this is the case for magnetic monopoles on
  $\bbR^3$ in which case these spaces are closely related to
  compactifications of the moduli spaces with the boundary faces mapping to
  lower charge idealized moduli spaces.
\end{abstract}
\maketitle

\tableofcontents

\section*{Introduction}

One of the natural tools for the study of asymptotically
translation-invariant phenomena on Euclidean spaces is the passage to the
`radial compactification' $X=\overline{\bbR^m} = \bbR^m \cup
S^{m-1}_\infty.$ This translates asymptotic behaviour to behaviour at the
boundary of $X$ and allows similar phenomena to be considered on arbitrary
compact manifolds with boundary, in terms of the intrinsic scattering
structure at the boundary. In this approach, emphasized in
\cite{MR95k:58168}, typical kernels and functions, such as Euclidean
distance, $d(z_1,z_2)^2= |z_1-z_2|^2$ which are quite singular near the
corner of the compact space $X\times X$ are resolved to `normal crossings',
i\@.e\@.~conormal singularities, by lifting to the the scattering stretched
product $X^2_{\scat}.$ This space is obtained by iterated real blow-up of $X^2.$
The corresponding triple product $X^3_{\scat}$ has also been discussed and
here we consider the `higher scattering products' of an arbitrary compact
manifold with boundary. These inherit the permutation invariance of $X^n$
and, apart from their construction, the most important result here is
that the projections onto smaller products also lift to be smooth
b-fibrations, giving
\begin{equation}
\xymatrix{
X^n_{\scat}\ar@<2ex>[r]_-{\cdots}
\ar@<-2ex>[r]^-{\cdots}&
X^{n-1}_{\scat}\quad \ldots\quad
X^4_{\scat}\ar@<1.5ex>[r]\ar@<.5ex>[r]\ar@<-.5ex>[r]\ar@<-1.5ex>[r]
&X^3_{\scat}\ar@<1ex>[r]\ar[r]\ar@<-1ex>[r]
&X^2_{\scat}\ar@<.5ex>[r]\ar@<-.5ex>[r]
&X.
}
\label{24.5.2008.134}\end{equation}
It is our basic contention that the spaces $X^n_{\scat},$ despite the
apparent complexity or their definition, are universal for asymptotically
translation-invariant phenomena.

There is of course a relation between the spaces considered here and the
seminal work of Fulton and MacPherson, \cite{MR95j:14002}. This
relationship is strongest at a combinatorial level but the differences are
also quite substantial. Apart from the distinction between real and complex
spaces in the two settings, it should be noted that all the blow ups
carried out here are in the boundary. As a result the space $X^n_{\scat}$
can be retracted to $X^n$ and in this sense the topology has not changed at
all. These spaces are designed to give `more room' for geometric and
analytic objects.

One justification for the introduction of these `higher' stretched products
is that they are anticipated to serve as at least combinatorial models for
a natural compactification of the moduli space of magnetic monopoles on
$\bbR^3$ and allow the detailed asymptotic description of the
hyper-K\"ahler metric. This will be shown elsewhere and is closely related
to the existence of the maps in \eqref{24.5.2008.134}. Given the
permutation-invariance of the spaces, the existence of these maps can be
reduced to one case in each dimension and then corresponds to a commutative
diagramme
\begin{equation}
\xymatrix{
X^n_{\scat}\ar[r]\ar@/^2pc/[rr]^{\pi_{n,\scat}}\ar[d]
&
X^{n-1}_{\scat}\times X\ar[r]
&
X^{n-1}_{\scat}\ar[d]\\\
X^n\ar^{\pi_n}[rr]
&
&
X^{n-1}
}
\label{ScatProd.102}\end{equation}
where $\pi_{n,\scat}$ is the new map and $\pi_n$ is projection off the last
factor. In fact this lifted map is a b-fibration, meaning in particular
that push-forward under it of a function with complete asymptotic expansion
(so in essence `smooth' up to the boundary faces) again has such an
expansion.

In the rest of this Introduction, we give an outline of this construction,
deferring the proofs to the main body of the article. Given their
fundamentally combinatorial nature the constructions here also extend to
the more general `fibred-cusp' configuration spaces $X^n_{\Phi},$
corresponding to the double and triple spaces introduced by Mazzeo and the
first author in \cite{MR2000m:58046}. These arise naturally when $\pa X$
comes with a fibration over a space $Y.$ The scattering case appears when
this fibration is the identity map $\pa X \longrightarrow\pa X.$

The main issue in what follows is the fundamental fact about iterated
blow-ups, which is that {\em different orders generally lead to
  non-diffeomorphic spaces}.  Thus while it is clear which submanifolds of
the boundary must be blown up in the construction of $X^n_{\scat},$ the
order in which this is done has to be specified. Moreover, having
determined this order, the existence of the map $\pi_{n,\scat}$ in
\eqref{ScatProd.102} corresponds to the possibility of obtaining the same
space by a performing these blow ups in a different order. In a manifold
with corners, $M,$ such as $X^n$ and the manifolds obtained by subsequent blow
up from it, admissible `centres' of blow up, $H\subset M,$ which is to say
manifolds which have a collar neighbourhood, are called p-submanifolds (for 
product-) and are always required here to be closed. The blow up of a
p-submanifold is then always possible and is denoted, with its blown-down map,
\begin{equation}
\beta: [M;H]\longrightarrow M.
\label{ScatProd.171}\end{equation}

Thus we are interested in the circumstances in which two
p-submanifolds $H_1$ and $H_2$ of a manifold with corners \emph{commute} in
the sense that the natural identification of the complement of the preimage
of $H_1\cup H_2$ in the blown up spaces extends to a smooth diffeomorphism
allowing us to identify
\begin{equation}
[M;H_1;H_2] = [M;H_2;H_1].
\label{e1.29.4.8}\end{equation}
Here, and throughout the paper, we have identified submanifolds with their
lifts to the blow-up. The lift of $H_2$ to $[X;H_1]$ is the closure in
$[X;H_1]$ of $\beta_1^{-1}(H_2\setminus H_1)$ if this is non-empty and
is $\beta_1^{-1}(H_2)$ in the opposite case (i.e. if $H_2\subset H_1$). 
In particular to conclude \eqref{e1.29.4.8} we need to know that
each lifts to be a p-submanifold under the blow-up of the other.  In fact, 
as is well-known, $H_1$ and $H_2$ commute in this sense if and only if they
are either transversal (including the case that they are disjoint) or
\emph{comparable}, meaning either $H_1\subset H_2$ or $H_2\subset H_1.$
To prove our results, we need to show that whole families of blow-ups
commute and this is where the combinatorial complexity lies.

\subsection{Boundary products} The `scattering structure' on a manifold
with corners can be identified with the intrinsic Lie algebra of
`scattering vector fields' consisting precisely of the products $fV$ where
$V\in\bV(X),$ meaning that $V$ is a smooth vector field which is tangent
to the boundary, and $f\in\CI(X)$ vanishes at the boundary. The larger Lie
algebra $\Vb(X)$ is the `boundary (b-) structure'; it can also be thought of as
representing the asymptotic multiplicative structure near the boundary or
geometrically as the `cylindrical end' structure on the manifold.

Not surprisingly then, our construction begins from the corresponding
$n$-fold stretched b-product $X^n_{\bo}.$ This resolves, near the diagonal
in the boundary, the pairwise distance functions (and of course much more
besides) for a cylindrical-end metric, also called a b-metric, on $X.$
Assuming, for simplicity and from now on, that $\partial X$ is connected,
$X^n_{\bo}$ is obtained from $X^n$ by the blow up of all boundary faces.

In order to describe the construction, consider the collection of all
boundary faces, $\cM(X^n),$ of $X^n.$ Blowing up boundary hypersurfaces of
a manifold with corners does noting, so let $\cB_{\bo}\subset\cM(X^n)$ be
the subset of boundary faces of $X^n$ of codimension at least $2;$ if $\pa
X$ is not connected $\cB_{\bo}$ is a smaller collection.

As standard notation we shall not distinguish between 
boundary faces of $X^n$ and their lifts to blown up versions of the
manifold, except where absolutely necessary.   
\begin{definition}
The b-stretched products of $X,$ $X^n_{\bo}$ are defined to be 
\begin{equation}
X^n_{\bo} = [X^n; \cB_{n}; \cB_{n-1}; \ldots ; \cB_{2}].
\label{ScatProd.172}\end{equation}
where $\cB_{r}\subset\cM(X^n)$ is the collection of boundary faces of
codimension $r.$ 
\end{definition}
\noindent Note the contracted notation in \eqref{ScatProd.172} for iterated
blow ups.  Since we have not specified an ordering of elements within each
of the families $\cB_r,$ it is implicit that the result does not depend on
these choices. In fact at the stage at which the elements of $\cB_{r}$ are
blown up they are disjoint so the order is immaterial and $X^n_{\bo}$ is
well-defined. It also follows from this that the permutation group lifts to
$X^n_{\bo}$ as diffeomorphisms.

Consider the analogue of \eqref{ScatProd.102} in this simpler setting:
\begin{equation}
\xymatrix{
X^n_{\bo}\ar[r]\ar@/^2pc/[rr]^{\pi_{n,\bo}}\ar[d]
&
X^{n-1}_{\bo}\times X\ar[r]
&
X^{n-1}_{\bo}\ar[d]\\\
X^n\ar^{\pi_n}[rr]
&
&
X^{n-1}
}
\label{ScatProd.104}\end{equation}
To show the existence of $\pi_{n,\bo}$ we divide $\cB_r$ into two pieces,
the {\em vertical} and {\em non-vertical} boundary faces (with respect to
the projection $\pi_n).$ Namely $\cB_r^{\ver}$ consists of those boundary
faces $B$ of $X^n$ of codimension $r$ such that the $n$-th factor of
$B$ is $X.$ Similarly, $\cB_r^{\nver}$ consists of those boundary faces $B$
of $X^n$ of codimension $r$ such that the $n$-th factor of $B$ is $\pa
X.$ Thus $B\in\mathcal{B}^{\ver}_{\bo}$ if $B=B'\times X$ with
$B'\in\mathcal{M}(X^{n-1}),$ i\@.e\@. $B=\pi_n(B)\times X.$ Otherwise the
$n$th factor is necessarily $\pa X$ and then
$B\in\mathcal{B}^{\nver}_{\bo}.$ Equivalently, the vertical boundary faces
are those that are unions of fibres of $\pi_n.$ 

Now observe that
\begin{equation}
X^{n-1}_{\bo}\times X=[X^n;\cB^{\ver}_{n-1};\ldots;\cB^{\ver}_{2}]
\label{24.5.2008.106}
\end{equation}
since the last factor of $X$ is unchanged throughout. Thus the existence of
$\pi_{n,\bo}$ in \eqref{ScatProd.104} follows if we show that the
non-vertical boundary faces can all be commuted to come last and hence that
$$
X^n_{\bo}=[X^{n-1}_{\bo} \times X;\cB_{n}^\nver;\cB_{n-1}^\nver;\ldots;\cB_{2}^\nver],
$$
so exhibiting $X^n_{\bo}$ is an iterated blow-up of $X^{n-1}_{\bo}\times X.$

\subsection{Multi-diagonals}\label{Sect-Md}

The space $X^n_{\scat}$ is constructed from $X^n_{\bo}$ by the blow up of
the intersections of the (lifts of the) `multi-diagonals' in $X^n$ with the
various boundary components of $X^n_{\bo},$ again with strong restrictions
on the order in which this is done. The total diagonal $\Diag\subset X^m$
is diffeomorphic to $X$ and the \emph{simple diagonals} in $X^n$ are the
images of $\Diag(X^m)\times X^{n-m}$ under the factor permutation maps. The
multi-diagonals (later called simply diagonals) are the intersections of
these simple diagonals. Since we are assuming the boundary of
$X$ to be connected, we can identify a simple diagonal, involving equality
for some collection of factors, with the boundary face $B\in\cB_{(2)}$
which has a factor of $\pa X$ exactly in each of the factors involving
equality. Then the multi-diagonals $D_{\fb}$ can be identified with
transversally-intersecting subsets $\fb\subset\cB_{(2)},$ meaning that
different elements do not have a boundary factor in common.

It is important to understand that the diagonals are not p-submanifolds in
$X^n.$ Nor, in general, are their boundary faces (which is what we are most
interested in). Indeed there are always points analogous to the boundary of
the diagonal in $[0,1]^2$ and hence they do not have a local product structure
consistent with that of the manifold. However, the effect of the
construction above is to resolve these `singularities'.
\begin{equation}
\text{The lift from }X^n\text{ to }X^n_{\bo}\text{ of each }D_{\fb}\text{
  is a p-submanifold}
\end{equation}
Now, the lift of $D_{\fb}$ will generally meet many boundary faces of 
$X^n_{\bo}.$ In particular it meets the lift of every boundary face
$B\in\cB_{(2)}$ with $B\subset\cap\fb$ and these intersections are the
`boundary diagonals' which are to be blown up. Thus,
(recalling that the lift of $B\in\cB_{(2)}$ to $X^n_{\bo}$ is also denoted
simply as $B)$ set
\begin{equation}
\caH =\{H_{B,\fb}=B\cap D_{\fb}\subset X^n_{\bo};B\in \cB\Mwith B \subset\cap\fb\}.
\end{equation}

To blow up all these submanifolds we need to choose an order and this is
required to respect the `lexicographic' partial ordering of $\caH$  
\begin{equation}
H_{A,\fa} \leq H_{B,\fb}\Longleftrightarrow D_{\fa}\subset D_{\fb}\Morif
D_{\fa}=D_{\fb}\Mthen A\subset B.
\label{ScatProd.174}\end{equation}
Then the scattering configuration space $X^n_{\scat}$ is defined to be
\begin{equation}
X^n_{\scat}=[X^n_{\bo};\caH].
\label{ScatProd.175}\end{equation}
with blow up in such an order. Of course it needs to be checked that the
result is independent of the choice of order consistent with
\eqref{ScatProd.174}. As in the case of $X^n_{\bo}$ follows from the fact
that the changes in order correspond to transversal intersections 
(including disjointness).

As noted above, these results are already known in the cases $n=2$ and
$n=3.$ Two new phenomena make the general case more complicated. The first
is the necessity to blow-up multi-diagonals with the first non-trivial
multi-diagonal occuring for $n=4.$ The second is the issue of the ordering
of $\caH.$

\subsection{Stretched projections}

As noted above, the existence of the `stretched projections' in
\eqref{24.5.2008.134} and \eqref{ScatProd.102} is the crucial property of
the spaces $X^n_{\scat}.$ The proof of their existence is in essence the
same as that outlined above, see \eqref{24.5.2008.106}, for the simpler
$\pi_{n,\bo}$ maps but the required commutation results are necessarily
more intricate. In particular we need to consider various spaces
intermediate between $X^n$ and $X^n_{\scat}$ which arise in this
argument. For instance, the p-submanifold $H_{A,\fa}\subset X^n_{\bo}$
actually makes sense already in $[X^n;\fa]$ from which it can be lifted
under further blow up. These issues are discussed extensively in the
article below, here we ignore such niceties to explain the procedure used
later.

The notion of vertical and non-vertical for boundary faces with respect
to the last factor discussed above extends to the boundary diagonals
$H_{A,\fa}.$ First extend it to transversal subsets $\fa\subset\cB_{(2)}$
where $\fa$ is vertical if and only if $\cap\fa$ is vertical.
Thus
\begin{equation}
\begin{gathered}
\caH_{\ver,\ver} = 
\{ H_{A,\fa} \in \caH;\fa\Mand B\text{ are vertical}\},\\
\caH_{\nver,\ver} =
\{ H_{A,\fa} \in \caH;\fa\text{ is vertical but }B\text{ is non-vertical}\},\\
\caH_{\nver,\nver} =
\{ H_{A,\fa} \in \caH;\fa\Mand B\text{ are non-vertical}\}.
\end{gathered}
\end{equation}
Notice that $A\subset\cap\fa$ in the definition of $H_{A,\fa}$ so if $\fa$
is non-vertical, then so is $A.$ From the definitions, it is clear that we have
\begin{equation}
X^{n-1}_{\scat} \times X
= [X^n; \cB^\ver; \caH_{\ver,\ver}]
\end{equation}
with the appropriate order on the blow ups. So the task is to recognize
$X^n_{\scat}$ as an iterated blow-up of this space. To do this, we first
show that all the `purely non-vertical' boundary diagonals can be blown
down so that, as always with appropriate orders on the collections of centres,
\begin{equation*}
\begin{gathered}
X^n_{\scat}
=[X^n_{\bo}; \caH_{\ver,\ver}\cup\caH_{\nver,\ver};\caH_{\nver,\nver}]\\
=[X^{n-1}_{\bo}\times X;\cB^{\nver};\caH_{\ver,\ver}\cup\caH_{\nver,\ver};
\caH_{\nver,\nver}].
\end{gathered}
\end{equation*}
Thus all the boundary faces of non-vertical diagonals are first removed.

To proceed further, we remove the `last' (which means originally largest)
boundary face from $B\in\cB^{\nver}$ by showing that it can be commuted
past all the subsequent boundary diagonals. Then all the $H_{B,\fa}$
corresponding to this boundary face are commuted out and blown down. Then
this procedure is iterated, at each step removing the last remaining
non-vertical boundary face and \emph{then} the boundary diagonals contained
in it.

The rest of this article is devoted to providing a rigorous discussion of
this outline. In \S\ref{mwc} material on manifolds with corners and blow up
is briefly recalled and in \S\ref{Bbu} the effect of the blow up of
boundary faces is considered. This is extended in \S\ref{Intord} to get a
basic result on the reordering of boundary blow up, which is used
extensively in the remainder of the article. In particular in
\S\ref{b-conf} the results described above for the boundary configuration
spaces $X^n_{\bo}$ are derived. The diagonals and their resolution is
examined in \S\ref{Multidiags} and the properties of these submanifolds are
slightly abstracted in \S\ref{D-coll} to aid the discussion of iterative
blow up. Collections of boundary diagonals are described in \S\ref{b-diag}
and used to construct the spaces $X^n_{\scat}$ in \S\ref{Scatcon}. Three
results on the reordering of blow ups of boundary diagonals are given in
\S\ref{Reorder} and these are used to carry out the construction of the
stretched projections in \S\ref{Sect-Ssp}. A simple application of these
spaces in \S\ref{Vspaces} is inserted to indicate why these resolutions
should prove useful.

\paperbody
\section{Manifolds with corners}\label{mwc}

Since we make heavy use of conventions for manifolds with corners , we give
a brief description of the basic results which are used below. These can
also be found in the \cite{daomwc}.

\subsection{Definition and boundary faces} By a manifold with corners we
shall mean a space modelled locally on products
$[0,\infty)^k\times\bbR^{n-k}$ with smooth transition maps (meaning they
  have smooth extensions across boundaries.) For such a space, $M,$
  $\CI(M)$ is well-defined by localization and at each point the boundary
  has a definite codimension, corresponding to the number, $k$ of functions
  in $\CI(M)$ vanishing at the point which are non-negative nearby and have
  independent differentials. We will insist that the boundary
  hypersurfaces, the closures of the sets of codimension $1,$ be
  embedded. This corresponds to the existence of functions $\rho
  _i\in\CI(M)$ which are everywhere non-negative and have $d\rho_i\not=0$
  on $\{\rho _i=0\}$ and such that $\pa M=\{\prod_i\rho _i=0\}.$ The
  connected components of the sets $\{\rho _i=0\}$ are the boundary
  hypersurfaces, the collection of which is denoted $\mathcal{M}_1(M).$ The
  components of the intersections of these hypersurfaces form boundary
  faces, all are closed and are the closures of their interiors the points
  of which have fixed codimension; thus $\mathcal{M}_k(M)$ consists of all
  the (connected) boundary faces of (interior) codimension $k$ and
  $\mathcal{M}_{(k)}(M)$ denotes the collection of codimension at least
  $k.$ By convention, we shall include $M\in\cM(M)$ as a `boundary face of
  codimension zero'.

Near a point of $M,$ where the boundary has codimension $k,$ it is
generally natural to use coordinates adapted to the boundary. That is,
local coordinates $x_i\ge0,$ $i=1,\dots,k$ and $y_j,$ $j=1,\dots,n-k$ where
the boundary hypersurfaces through the given point are the $\{x_i=0\}.$

If $B_1,$ $B_2\in\mathcal{M}(M)$ then their intersection is a boundary face
(possibly empty) but their union is not. However the union is contained in
a smallest boundary face which we will denote 
\begin{equation}
B_1\dotplus B_2=\bigcap\{B\in\mathcal{M}(M);B\supset B_1\cup B_2\}.
\label{24.5.2008.144}\end{equation}

\subsection{p-submanifolds} Embedded submanifolds of a manifold with
corners can be rather more complicated locally than in the boundaryless
case. The simplest type is a p-submanifold. This is a closed subset
$Y\subset M$ which has a local product decomposition near each point,
consistent with a local product decomposition of $M.$ An interior
p-submanifold (not necessarily contained in the interior) is distinguished
by the fact that locally in a neighbourhood $U$ of each of its points there
are $l$ independent functions $Z_i\in\CI(U)$ which define it and which are
independent of the local boundary defining functions, i\@.e\@.~it is
defined by the vanishing of interior coordinates. A general p-submanifold
is an interior p-submanifold of a boundary face. Any p-submanifold $Y$ of
$M$ can be locally put in {\em standard form} near a point $p$ in the sense
that there are adapted coordinates $x_i,$ $y_j,$ based at $p$ such that in the
coordinate neighbourhood $U$ 
\begin{equation}
U\cap Y = \{(x,y) \in U; x_j=0,\ 1\le j\le l,\ y_i = c_i\mbox{ for }i\in I\}
\label{24.5.2008.161}\end{equation}
where $I$ is some subset of the index set for interior coordinates and the
$c_i$ are constants.

\subsection{Blow-up and lifting of manifolds and maps}

The (radial) blow up of a p-submanifold is always well-defined and yields a
new compact manifold with corners $[M;Y]$ with a blow-down map $\beta
:[M;Y]\longrightarrow M$ which is a diffeomorphism from the complement of
the preimage of $Y$ to the complement of $Y.$

\begin{lemma}\label{ScatProd.61} Under blow up of a boundary face of a
manifold with corners, any p-submanifold $H$ lifts to a p-submanifold which
is contained in the lift of its boundary-hull, the smallest boundary face
containing $H.$
\end{lemma}

Note that the lift, also called the `proper transform', of a subset
$S\subset M$ under the blow up of a centre, $B\subset M,$ which is required
to be p-submanifold for this to make sense, is the subset of $[M;B]$ which
is \emph{either} the inverse image $\beta^{-1}(S),$ if $S\subset B,$ or
else the closure in $[M;B]$ of $\beta ^{-1}(S\setminus B)$ if it is not. Of
course this is a useful notion only for sets which are `well-placed'
with respect to $B.$

\subsection{Comparable and transversal submanifolds}

In the sequel the intersection properties of boundary faces, and later
manifolds related to multi-diagonals, play an important role. The manifolds
we consider here will always intersect cleanly, in the sense of Bott. That
is, at each point of intersection they are modelled by linear spaces and
their intersection is therefore also a manifold. This is immediate from the
definition for boundary faces and almost equally obvious for the diagonal-like
manifolds we consider later. We will say that two such manifolds, and this
applies in particular to boundary faces, $H_1$ and $H_2,$ of $M$ are
\begin{itemize}
\item Comparable if $H_1\subset H_2$ or $H_2\subset H_1.$
\item Transversal, written $H_1\pitchfork H_2$ if $N^*B_1$ and $N^*B_2$ are
  linearly independent at each point of intersection.
\item Neither comparable nor transversal, abbreviated to `\ncnt' otherwise.
\end{itemize}

More generally a collection of submanifolds $H_i,$ $i=1,\dots,J$ is
transversal if at every point $p$ of intersection of at least two of them, the
conormals $N^*_pH_i$ for those $i$ for which $p\in H_i$ are
independent. This in particular implies that the intersection is a manifold.

If $B_1\pitchfork B_2$ are two p-submanifolds of $M$ then the lift of $B_2$
to $[M;B_1]$ which is defined above to be the closure of $B_2\setminus B_1$
in the blown up manifold, is also equal to the inverse image, $\beta
^{-1}(B_2).$

\subsection{b-maps and b-fibrations} A general class of maps between
manifolds with corners which leads to a category are the b-maps. These are
maps $f:M\longrightarrow M'$ which are smooth in local coordinates and have the
following additional property. Let $\rho '_i$ be a complete collection of
boundary defining functions on $M'$ and $\rho _j$ a similar collection on
$M.$ Then there should exist non-negative integers $\alpha_{ij}$ and
positive functions $a_i\in\CI(M)$ such that

\begin{equation}
f^*\rho' _i= a_i\prod_{j}\rho _j^{\alpha _{ij}}.
\label{24.5.2008.116}\end{equation}
Such a map is b-normal if for each $j,$ $\alpha_{ij}\not=0$ for at most one
$i.$ This means that no boundary hypersurfaces of $M$ is mapped completely
into a boundary face of codimension greater than $1$ in $M'.$

The real vector fields on $M$ which exponentiate locally to diffeomorphisms
of $M$ are the elements of $\Vb(M),$ meaning smooth vector fields on $M$
which are tangent to all boundary faces. These form all the smooth sections
of a natural vector bundle $\bT M$ over $M$ and each b-map has a
b-differential at each point $f_*:\bT_pM\longrightarrow \bT_{f(p)}M'.$ A
b-map is said to be a b-submersion if this map is always
surjective. Blow-down maps are always b-maps and for boundary faces they
are b-submersions but not for other p-submanifolds. A b-map which is both a
b-submersion and b-normal is a b-fibration; blow down maps are never
b-fibrations. With the notion of smoothness extended to include classical
conormal functions on a manifold with corners, b-fibrations are the
analogues of fibrations in the sense that such regularity is preserved
under push-forward.

\section{Boundary blow up}\label{Bbu}

Each boundary face $B\in\mathcal{M}(M)$ of a manifold with corners is, as a
consequence of the assumption that the boundary hypersurfaces are embedded,
a p-submanifold. Thus, it is always permissible to blow it up. If $B$ has
codimension one (or zero), this does nothing. If $k=\codim(B)\ge2,$
i\@.e\@.~$B\in\mathcal{M}_{(2)}(M),$ one gets a new manifold with corners,
$[M;B],$ with one new boundary hypersurface $\ff([M;B])$ -- which is the
positive $2^{k}$th part of a trivial $(k-1)$-sphere bundle over $B,$ $k$
being the codimension of $B.$ This fractional-sphere bundle is the
inward-pointing part of the normal sphere bundle to $B$ and is trivialized
by the choice of a defining function for each of the boundary hypersurfaces
of $M$ containing $B.$ More generally, the other boundary hypersurfaces of
$[M;B]$ are in 1-1 correspondence with the boundary hypersurfaces of $M$
where again the assumption that the boundary hypersurfaces are embedded
means that the connectedness cannot change on blow up of $B.$ More
generally, the boundary faces of $[M;B]$ not contained in $\ff([M;B])$ are
the lifts (closures of inverse images of complements w.r.t.\ $B)$ of the
boundary faces of $M$ not contained in $B.$ The boundary faces of $[M;B]$
contained in $\ff([M;B])$ are either preimages (also called lifts) of
boundary faces of $B$ -- hence are the restrictions of the
fractional-sphere bundles to boundary faces of $B$ -- or else are proper
boundary faces of the fractional balls over one of these faces (including
of course $B$ itself). The latter ones are `new' boundary faces, not the
lifts of old ones. We identify, at least notationally, each boundary face
$B'$ of $M$ with its lift to a boundary face of $[M;B],$ even though the
latter may be either a blow-up of $B',$ if $B'$ is not contained in $B$
initially, or a bundle over $B'$ if it is -- in which case the dimension
has increased.

For later reference we examine the effect of the blow up of a boundary face
on the intersection of two others.

\begin{lemma}\label{ScatProd.34} Consider two distinct boundary faces $B_1$
  and $B_2$ in a manifold with corners $M$ and their lifts to $[M;B]$ where
  $B\in\mathcal{M}_{(2)}(M):$
\begin{enumerate}
\item[(i)] If $B_1\pitchfork B_2$ then their lifts are transversal in
  $[M;B];$ they are disjoint if and only if $B_1\cap B_2\subset B$ but
  $B_1\setminus B\not=\emptyset$ and $B_2\setminus B\not=\emptyset.$
\item[(ii)] If $B_1\subset B_2$ in $M,$ then their lifts to $[M;B]$ are
  never disjoint, they are comparable in $[M;B]$ if and only if 
\begin{equation}
\begin{gathered}
B_1\setminus B\not=\emptyset,\\
B_1\subset B\Mand B\pitchfork B_2\\
\Mor B_2\subset B,
\end{gathered}
\label{ScatProd.36}\end{equation}
the lifts are transversal if  
\begin{equation}
B_1\subset B\subsetneq B_2
\label{ScatProd.38}\end{equation}
and are otherwise \ncnt\ (that is if $B=B_2$ or $B_1\subset B$ but $B$
and $B_2$ are neither transversal nor is $B_2\subset B).$ 
\item[(iii)] If $B_1$ and $B_2$ are \ncnt\ in $M$ then their lifts are
  never comparable, are disjoint in $[M;B]$ if and only if
\begin{equation}
B_1\cap B_2\subset B\text{ and both }B_1\setminus
B\not=\emptyset\Mand B_2\setminus B\not=\emptyset,
\label{ScatProd.35}\end{equation}
they lift to meet transversally if and only if (using
\eqref{24.5.2008.144}) 
\begin{equation}
\Meither B_1\subset B\subsetneq B_1\dotplus B_2\Mor B_2\subset
B\subsetneq B_1\dotplus B_2
\label{24.5.2008.145}\end{equation}
and otherwise their lifts are \ncnt.
\end{enumerate}
\end{lemma}

\begin{proof} When discussing the local effect of the blow up of a boundary
  face we may always choose adapted coordinates $x_i\ge0,$ $y_j;$ in fact
  the `interior coordinates' $y_j$ play no part in the discussion here.

There exist subsets $I$, $I_1$, $I_2$ of $\{1,\ldots,k\}$ such that
\begin{equation}\label{e1.16.7.8}
B_i = \{x_j = 0: j \in I_i\},\; i=1,2,\;
B = \{x_j = 0: j \in I\}.
\end{equation}
The three parts of the Lemma correspond to the mutually exclusive
cases where (i) $I_1\cap I_2 = \emptyset$; (ii) $I_1\subset I_2$ or
$I_2\subset I_1$; (iii) neither of these conditions hold. Of course we only
need consider the first of the cases in (ii).

Near any point of the front face of $[X;B]$ there are adapted coordinates
with the interior coordinates, $y_j,$ lifted from $X$, and the boundary
coordinates $x_j$ replaced by 
\begin{equation}
t_B=\sum\limits_{i\in I}x_i,\Mand
t_j =\begin{cases}
x_j/t_B\mbox{ if }j\in I\\
x_j \mbox{ if }j\not\in I.
\end{cases} 
\end{equation}
Note that 
\begin{equation}
\sum\limits_{j\in I}t_j=1.
\label{24.5.2008.159}\end{equation}

Here $t_B$ defines the new front face, i\@.e\@.~the lift of $B.$
The lifts $\tB_i$ of the $B_i$ to $[X;B]$ are given by
\begin{equation} \label{e2.16.7.8}
\tB_i=\begin{cases}
\{t_B=0,\; t_j=0,\; j \in I_i\setminus I\}&\Mif B_i\subset B
\Longleftrightarrow I\subset I_i\\
\{t_j=0,\; j \in I_i\}
&\Mif B_i\setminus B \neq \emptyset,\ \Longleftrightarrow
I\not\subset I_i.
\end{cases}
\end{equation}
 
We can now examine the intersection properties of $\tB_1$ and
$\tB_2$.
\begin{enumerate}
\item[(a)] First assume $B_1$ and $B_2$ are both contained in $B$. Then
from \eqref{e2.16.7.8} it is clear that the lifts $\tB_1$ and $\tB_2$
are never transversal, since $t_B,$ the defining function of the front
face, vanishes on the both of them. They are clearly comparable if and
only $B_1$ and $B_2$ are comparable and are otherwise \ncnt.
\item[(b)] Second, suppose $B_1\subset B$ but $B_2\setminus B\not=\emptyset$.
From \eqref{e2.16.7.8} we see that $\tB_1$ and
$\tB_2$ meet transversally if and only if $I_1\setminus I$ and $I_2$ are
disjoint. The lifts can only be comparable if $\tB_1 \subset \tB_2$
and this occurs if and only if $I_1\setminus I$ contains $I_2.$  Otherwise,
$\tB_1$ and $\tB_2$ are \ncnt.
\item[(c)] Finally, suppose that $B_1\setminus B$ and $B_2\setminus B$ are both
non-empty. Then each $\tB_i$ is given by the vanishing of the $t_j$ for
$j\in I_i.$ Now $I_1\cup I_2 \supset I$ if and only if $B_1\cap B_2 \subset B.$
Hence $\tB_1$ and $\tB_2$ are disjoint in this case in view of
\eqref{24.5.2008.159}. Otherwise, the transversal, comparable or \ncnt
according as this is the case for the original submanifolds $B_1$ and
$B_2.$
\end{enumerate}

It is now a simple matter to use these observations to prove the
Lemma. Consider part (i) in which $B_1$ and $B_2$ are transversal, or
equivalently, $I_1$ and $I_2$ are disjoint. Running through cases above,
(a) cannot arise and in (b) and (c) $\tB_1$ and $\tB_2$ are transversal and
are disjoint exactly as claimed.

Next consider part (ii) of the Lemma: without loss of generality, suppose
$B_1\subset B_2$, so $I_1\supset I_2.$ Then in case (a) the lifts are
comparable if both are contained in $B$ and according to (c) they are
comparable if both are {\em not} contained in $B.$ Comparable lifts arise
in case (b) if $I_2\subset I_1\setminus I,$ or equivalently if $I$ and
$I_2$ are disjoint subsets of $I.$ Thus the lifts are comparable in this
case if $B_1\subset B_2,$ $B_1\subset B,$ $B$ and $B_2$ are transversal.

We also see from (b) that the lifts of $\tB_1$ and $\tB_2$ are
transversal if and only if $I_1\cap I_2 \setminus I\cap I_2 = \emptyset$ or
equivalently if $I_2\subset I.$ This proves \eqref{ScatProd.38}.

Finally consider part (iii) of the Lemma. Under case (a) the lifts are
always \ncnt. Under case (b), the lifts are transversal if and only if 
$(I_1\setminus I)\cap I_2 = \emptyset,$ that is, if and only if $I_1\cap
I_2\subset I\cap I_2.$ Since $I\subset I_1$, this just means that
$I_1\cap I_2 \subset I$ which gives \eqref{24.5.2008.145}.  Otherwise,
they are \ncnt.  Under (c), the lifts are \ncnt unless the $B_1\cap
B_2\supset B,$ in which case they are disjoint, giving
\eqref{ScatProd.35}.
\end{proof}

We are most interested in the transversality of the intersections of the
lifts, meaning either they are disjoint or meet transversally.

\begin{corollary}\label{24.5.2008.143} If $B_1$ and $B_2$ are distinct
  boundary faces of a manifold with corners and $B\in\mathcal{M}_{(2)}(M)$
  then $B_1$ and $B_2$ are (i\@.e\@. lift to be) transversal in $[M;B]$ if
  and only if they are initially transversal or if not then
\begin{equation}
B_1\subset B\subsetneq B_1\dotplus B_2\Mor
B_2\subset B\subsetneq B_1\dotplus B_2.
\label{24.5.2008.146}\end{equation}
Two boundary faces lift to be disjoint if  
\begin{equation}
B_1\cap B_2\subset B\text{ and both }B_1\setminus
B\not=\emptyset\Mand B_2\setminus B\not=\emptyset.
\label{24.5.2008.147}\end{equation}
\end{corollary}
\noindent Note that if $B_1\pitchfork B_2$ then $B_1\dotplus B_2=M.$ 

\section{Intersection-orders}\label{Intord}

Since boundary faces lift to boundary faces under blow up of any boundary
face, any collection, $\mathcal{C}\subset\mathcal{M}(M),$ in any manifold
with corners, can be blown up in any preassigned order, leading to a
well-defined manifold with corners. Of course this actually means that
after the first blow-up the lift of the second boundary face is blown up,
and so on. Let the order of blow up be given by an injective function
\begin{equation}
o:\mathcal{C}\longrightarrow \bbN
\label{ScatProd.33}\end{equation}
where for simplicity we also assume that the range is an interval $[1,N]$ in
the integers. Denote the total blow-up as $[M;\mathcal{C},o];$ in general
the choice of order does make a difference to the final result but the
interior is always canonically identified with the interior of $M.$

From now on we will assume that the initial collection of boundary faces,
$\mathcal{C},$ of $M$ which are to be blown up is closed under
non-transversal intersection in $M:$
\begin{equation}
B,\ B'\in\mathcal{C}\subset\mathcal{M}_{(2)}(M)\Longrightarrow B\pitchfork
B'\Mor B\cap B'\in\mathcal{C}.
\label{ScatProd.43}\end{equation}
It should be noted that this is a condition in $M$ and can fail for the
lifts under blow up of boundary faces in that the intersection of the lifts
may not be equal to the lift of the intersection.

\begin{lemma}\label{24.5.2008.148} If $\mathcal{C}\subset\mathcal{M}(M)$ is
  closed under non-transversal intersection then it is a disjoint union of
  collections $\mathcal{C}_i\subset\mathcal{C}$ which are also closed
  under non-transversal intersection, each contain a unique minimal
  element, and are such that all intersections between elements of 
  different $\mathcal{C}_i$ are transversal.
\end{lemma}
\noindent These $\mathcal{C}_i$ may be called the transversal components of
$\mathcal{C}.$

\begin{proof} Consider the minimal elements $A_i\in\mathcal{C},$ those
  which contain no other element. These must intersect transversally, since
  otherwise the intersection would be in $\mathcal{C}$ and they would not
  be minimal. Then set $\mathcal{C}_i=\{F\in\mathcal{C};F\supset A_i\};$
  these are certainly closed under non-transversal intersection. On the
  other hand the defining functions for elements of $\mathcal{C}_i$ are
  amongst the defining functions for $A_i.$ It follows that the different
  $\mathcal{C}_i$ are disjoint, since their elements cannot have a defining
  function in common, and also that intersections between their elements
  are transversal.
\end{proof}

\begin{definition}
An order $o$ on a collection $\cC$ of boundary faces is an
\emph{intersection-order} if for any pair $B_1$ and $B_2\in\mathcal{C}$
which are not transversal or comparable, $B_1\cap B_2$ comes earlier than
at least one of them, i\@.e\@.
\begin{equation}
B_1,\ B_2\in\mathcal{C} \Longrightarrow B_1\pitchfork B_2\Mor o(B_1\cap
B_2)\le\max(o(B_1),o(B)_2)).
\label{ScatProd.47}\end{equation}
\end{definition}
Of course if $B_1$ and $B_2$ are comparable then the intersection
is equal to one of them so \eqref{ScatProd.47} is automatic. On the other
hand if $B_1$ and $B_2$ intersect non-transversally then
\eqref{ScatProd.47} implies that the intersection comes strictly before the
second of them with respect to the order. We will not repeatedly say that
$\mathcal{C}$ is closed under non-transversal intersection, just that it has an
intersection-order which is taken to imply the closure condition.

\begin{definition}
An order $o$ on a collection $\cC$ of boundary faces is a
\emph{size-order} if the codimension is weakly  
decreasing with the order, i.e.\
\begin{equation}
o(B_1) < o(B_2) \Longrightarrow \codim(B_1) \geq \codim(B_2).
\end{equation}
\end{definition}
Clearly a size-order is an intersection-order
since the intersection of \ncnt\ boundary faces necessarily has larger
codimension than either of them and so must occur first in the order of the
three.

\begin{lemma}\label{ScatProd.41} The iterated blow up in a manifold with
  corners $M$ of a collection of boundary faces $\mathcal{C},$ which is
  closed under non-transversal intersection, with respect to any two size
  orders gives canonically diffeomorphic manifolds, with the diffeomorphism
  being the extension by continuity from the identifications of the interiors.
\end{lemma}

\begin{proof} The first element, $B,$ in the order necessarily has maximal
  codimension so cannot contain any other. Thus all lifts of elements of
  $\cC'=\cC\setminus\{B\}$ are closures of complements with respect to $B;$
  their lifts therefore have the same dimension as before and hence in the
  induced order on $\cC'$ in $[M;B]$ the codimension is weakly decreasing.

Now, we proceed to show that the lift of the elements of $\cC'$ to $[M;B]$
is closed under non-transversal intersection. So, consider two distinct
elements $B_1,$ $B_2\in\cC'.$ If they are comparable then $B$ cannot
contain the smaller so by, Lemma~\ref{24.5.2008.148} they lift to be
comparable. If they are transversal then again by Lemma~\ref{ScatProd.34}
they lift to be transversal. Finally, suppose $B_1$ and $B_2$ are \ncnt. Since
\eqref{24.5.2008.145} cannot arise here, either \eqref{ScatProd.35} holds,
and hence $B_1\cap B_2=B$ and they lift to be disjoint, or else $B_1\cap
B_2\setminus B\not=\emptyset$ and they lift to be \ncnt\ with intersection
the lift of $B_1\cap B_2\in\cC'.$

Thus after the blow up of the first element of $\cC$ the remaining elements
lift to a collection of boundary faces closed under non-transversal
intersection and in size-order. Now we can proceed by induction on the
number of elements of $\cC$ and hence assume that we already know that the
result of the blow up of $\cC'$ in $[M;B]$ is independent of the
size-order. If $B$ is the only element of maximal codimension in $\cC$ the
result follows. If there are other elements of the same codimension then by
Lemma~\ref{24.5.2008.148} they meet $B$ transversally. Thus, the order of
$B$ and the second element can be exchanged. Applying discussion above
twice it follows that the same manifold results from blow up in any
size-order on $\cC.$
\end{proof}

We proceed to show that the the same manifold results from the blow up in
any intersection-order.

\begin{proposition}\label{ScatProd.1} The iterated blow up of $M,$
  $[M;\mathcal{C},o],$ of an intersection-ordered collection of boundary
  faces is a manifold with corners independent of the choice of
  intersection-order in the sense that different orders give canonically
  diffeomorphic manifolds, with the diffeomorphism being the extension by
  continuity from the identifications of the interiors.
\end{proposition}

\begin{proof} Let $o$ be the order in the form \eqref{ScatProd.33}. For
such an order we define the defect to be
\begin{equation}
d(o)=\sum\limits_{J\in\mathcal{C}}o(J)\max\{(\codim(J)-\codim(I))_+;o(I)<o(J)\}.
\label{1.1.2008.5}\end{equation}
Here the codimensions are as boundary faces of $M,$ not after
blow-up. 
Thus the defect is the sum over all sets of the maximum difference (if
positive) between the codimensions of the `earlier' sets and of that
set, weighted by the position of the set. Thus, for a size-order the
defect vanishes, 
because all these differences are non-positive, otherwise it is strictly
positive.

For a general intersection-order take the first set, with respect to the
order, $I,$ such that its successor, $J,$ had larger codimension in $M,$ and
consider the order $o'$ obtained by reversing the order of $I$ and $J.$ We
claim that this is an intersection-order and of strictly smaller
defect, and that $[M,\cC,o] = [M,\cC,o']$.

The last point will be checked first. Certainly $I$ cannot be the last
element with respect to the order. Note also that the boundary faces up to,
and including, $I$ are in size-order, by the choice of $I.$ Let
$\cP\subset\cC$ be the subcollection of strict predecessors of $I.$ Let
$M'$ be the manifold obtained from $M$ by blowing up $\cP.$ In order to be
able to commute the lifts, $\tI$ and $\tJ,$ of $I$ and $J$ to $M',$ we need
to rule out the possibility that they are \ncnt; Lemma~\ref{ScatProd.34}
will be used repeatedly for this.

Suppose first that $I$ and $J$ are comparable in $M.$ Then $J\subset I.$
According to Lemma~\ref{ScatProd.34}, such a comparable pair of
submanifolds can remain comparable, can become transversal, or can become
\ncnt in $M'.$ If they ever become transversal, then they remain so under
all subsequent blow-ups. Now comparable submanifolds $B_1\subset B_2$ can
only become \ncnt under a blow-up with centre $B$ satisfying $B_1\subset
B=B_2$ or $B_1\subset B,$ $B\setminus B_2\neq\emptyset$. Since the centres
of the blow-ups leading to $M'$ are all of smaller dimension than $I,$
which here plays the role of $B_2,$ we see that these conditions can never
be met by elements $B\in\cP.$ So if $I$ and $J$ are comparable in $M,$
their lifts to $M'$ cannot be \ncnt.

The only remaining possibility is that $I$ and $J$ are \ncnt in $M.$ In
this case, $I\cap J$ has strictly dimension than $I$ and so, because $o$ is
an intersection-order, this must be an element of $\cP.$ Moreover, because
$\cP$ is in a size-order, the lifts of $I$ and $J$ meet in the lift of
$I\cap J$ until this manifold is blown up, at which point they become
disjoint and then remain disjoint under all subsequent blow-ups. This shows
that $\tI$ and $\tJ$ commute in $M'.$

To show that $o'$ is an intersection-order, consider an \ncnt pair $A,$
$B.$ The only possibility of a failure of the intersection-order condition
for $o'$ is if $I$ was the intersection and $J$ the second element, in the
order, of such a pair. However this means that, initially in $M,$
$\codim(I)>\codim(J)$ and from the discussion above, this cannot occur.

Now, to compute the defect of $o'$ observe that each of the sets which came
after $J$ initially still have the same overall collection of sets
preceding them, and the same order, hence make the same contribution to
the defect. The same is true for the sets which preceded $I.$ Thus we only
need to recompute the contributions from $I$ and $J$ after reversal. In its
new position, $J$ has one less precedent, viz\@. $I$ now comes later, so the
set of differences $\codim(J)-\codim(I')$ where $o(I')<o'(J)$ is smaller
and the order of $J$ has gone down, so it makes a strictly smaller
contribution. The contribution of $I$ was zero before and is again zero,
since the only extra set preceding it, namely $J,$ has larger codimension
than it.

Thus this `move' strictly decreases the defect. Repeating the procedure a
finite number of times (note that after the first rearrangement, $J$
might well be the `new $I$') must reduce the defect to $0.$ Hence the blow
up for any intersection-order is (canonically) diffeomorphic to one for
which $\codim(B)$ is weakly decreasing, i\@.e\@. to a size-order and hence by
Lemma~\ref{ScatProd.41} all intersection-orders lead to the same blown-up
manifold.
\end{proof}

\begin{definition}\label{ScatProd.6} We denote by $[M,\mathcal{C}]$ the
iterated blow-up of any collection of boundary faces which is closed under
non-transversal intersection, with respect to any intersection-order.
\end{definition}

One simple rearrangement result which follows from this is:

\begin{lemma}\label{l1.21.7.8} Suppose $\cC_1\subset\cC\subset\cM(M)$ are
  both closed under non-transversal intersection, then there is an
  intersection order on $\cC$ in which the elements of $\cC_1$ come before
  all elements of $\cC\setminus\cC_1.$
\end{lemma}

\begin{proof} Let $o$ be a size-order on $\cC$ and consider the new order
  $o'$ on $\cC$ defined by $o'(B) = o(B)$ if $b\in \cC_1,$ $o'(B) = o(B) +
  N$ otherwise, where $N=\max(o).$ Then $o'$ has the desired property that every
  element of $\cC_1$ comes before every element of $\cC\setminus \cC_1.$
  Moreover, $o'$ must be an intersection-order. Thus we wish to show that
  if $B_1$ and $B_2$ are \ncnt then
  \begin{equation}\label{e11.21.7.8}
o'(B_1) < o'(B_2) < o'(B_1\cap B_2)
  \end{equation}
is not possible. This certainly cannot happen unless $B_1,$ $B_2 \in
\cC_1,$ $B_1\cap B_2\in \cC\setminus \cC_1,$ because $o'$ restricts to give
a size-order on each of $\cC_1$ and $\cC\setminus \cC_1.$ However, if $B_1$
and $B_2$ lie in $\cC_1$ then so does $B_1\cap B_2$ because $\cC_1$ is
closed under non-transversal intersection. Thus \eqref{e11.21.7.8} is
indeed impossible.
\end{proof}

\begin{corollary}\label{ScatProd.45} If
  $\mathcal{C}_1\subset\mathcal{C}_2$ are two collections of boundary faces 
  of $M,$ both closed under non-transversal intersection, then there is an
  iterated blow-down map 
\begin{equation}
[M;\mathcal{C}_2]\longrightarrow [M;\mathcal{C}_1].
\label{ScatProd.46}\end{equation}
\end{corollary}

\begin{proof} By the preceding lemma, there is an intersection-order on
  $\cC_2$ with respect to which all elements of $\cC_1$ come first. The
  existence of the blow-down map follows immediately from this.
\end{proof}

We will use the freedom to reorder blow ups frequently below. For instance
if $\mathcal{C}$ is closed under non-transversal intersection then any
given element is first or last in some intersection-order. In fact if the
elements are first given a size-order then any one element can be moved to
any other point in the order and the result is an
intersection-order. Another use of the freedom to change order established
above is to examine the intersection properties of boundary faces, as in
Lemma~\ref{ScatProd.34}, but after a sequence of boundary blow ups.

\begin{proposition}\label{ScatProd.39}
If $B_1,$ $B_2$ are distinct
  boundary faces of $M$ and $\mathcal{C}\subset\mathcal{M}(M)$ is closed under
  non-transversal intersection then
\begin{enumerate}
\item The lifts of $B_1$ and $B_2$ to $[M;\mathcal{C}]$ are disjoint if
  they are disjoint in $M$ or there exists $B\in\mathcal{C}$ 
  satisfying \eqref{24.5.2008.147}.
\item The lifts of $B_1$ and $B_2$ meet transversally in $[M;\mathcal{C}]$ if
they are transversal in $M$ or there exists $B\in\mathcal{B}$
  satisfying \eqref{24.5.2008.146}.
\item If $B_1\subset B_2$ in $M$ then this remains true for the lifts to
  $[M;\mathcal{C}]$ if
\begin{equation}
B\in\mathcal{C},\ B_1\subset B\Longrightarrow B_2\subset B\Mor
B_2\pitchfork B.
\label{ScatProd.40}\end{equation}
\end{enumerate}
\end{proposition}

\begin{proof} We can assume that $B_1$ and $B_2$ are both proper boundary
  faces. If there exists an element of $\mathcal{C}$ satisfying
  \eqref{24.5.2008.147} then, as noted above, there is an
  intersection-order on $\mathcal{C}$ in which a given element comes
  first. Lemma~\ref{ScatProd.34} shows that blowing it up first separates
  $B_1$ and $B_2$ which thereafter must remain disjoint. Thus shows the
  sufficiency of ~\ref{ScatProd.34}.

As above, if there is an element of $\mathcal{C}$ satisfying
\eqref{24.5.2008.146} then it can be blown up first in an
intersection-order which makes $B_1$ and $B_2$ transversal; then
Lemma~\ref{ScatProd.34} shows that persists under subsequent blow
up.

In the third part of the Proposition the sufficiency of the condition
follows immediately from Lemma~\ref{ScatProd.36} since the elements of
$\mathcal{C}$ of codimension two or greater containing $B_1,$ and so by
hypothesis either containing $B_2$ or transversal to it, form a collection
closed under non-transversal intersection. In fact this is separately true
of those containing $B_2$ and those which contain $B_1$ but are transversal
to $B_2$ since no intersection of the latter can contain $B_2.$ So all
these blow ups can be done first. In each case once the minimal element of
a transversal component is blown up the other elements do not contain $B_1$
so all blow ups preserve the inclusion of $B_1$ in $B_2.$
\end{proof}

\begin{lemma}\label{24.5.2008.150} For any two boundary faces $B_1,$
  $B_2\in\mathcal{C},$ with lifts denoted $\tilde B_i,$ $i=1,2$ it is
  always the case that 
\begin{equation}
\tilde B_1\cap \tilde B_2\subset \widetilde{B_1\cap B_2}\Min[M;\mathcal{C}].
\label{24.5.2008.151}\end{equation}
\end{lemma}

\begin{proof} Consider the decomposition
  $\mathcal{C}=\mathcal{C}'\cup\mathcal{C}''$ into the collections of
  elements which do not contain $B_1\cap B_2$ and those which do contain
  it. These must separately be closed under non-transversal
  intersection. Under blow up of an element of $\mathcal{C}',$ $B_1,$ $B_2$
  and $B_1\cap B_2$ all lift to the closure of their complements with
  respect to the centre so the lift of the intersection is the intersection
  of the lifts. Moreover the other elements of $\mathcal{C}'$ lift not to
  contain the intersection while the elements of $\mathcal{C}'$ lift to
  contain it. Thus after blowing up all the elements of $\mathcal{C}'$ we
  are reduced to the case that $\mathcal{C}'$ is empty, so we may assume
  that $B_1\cap B_2$ is contained in each element of $\mathcal{C}.$ Now,
  consider the decomposition of $\mathcal{C}$ as in
  Lemma~\ref{24.5.2008.148}. Consider the effect of the blow up of the
  minimal element $A_1\in\mathcal{C}_1.$ Now the lift of $B_1\cap B_2$ to
  $[M,A_1]$ is its preimage, the lifts of $B_1$ and $B_2$ depend on whether
  they are, or are not, contained in $A_1$ but in any case
  \eqref{24.5.2008.151} holds after this single blow up. If $A_1$ contains
  neither $B_1$ nor $B_2$ then by \eqref{24.5.2008.147} the lifts are
  disjoint and we need go no further. On the other hand if $A_1\supset
  B_1\cup B_2$ then all three manifolds lift to their preimages and
  equality of intersection of lifts and the lift of the intersections
  persists. The lifts of the other elements of $\mathcal{C}_1$ contain no
  fibres of the front face of $[M;A_1]$ over $A_1$ and so cannot contain
  the intersection. Hence these blow ups again preserve the equality. The
  only case remaining is where $A_1$ contains one, but not both, of $B_1$
  and $B_2.$ We can assume that $A_1\supset B_1$ and then all the elements
  of $\mathcal{C}_1$ satisfy this. Let
  $\mathcal{C}_1=\mathcal{C}_1'\cup\mathcal{C}_1''$ be the decomposition
  into those (before blow up of $A_1)$ which do not contain $B_2$ and do
  contain $B_2,$ where the second collection may be empty. Blowing up in
  size-order for each of these subcollections, observe that after the blow
  up of $A_1,$ the other elements of $\mathcal{C}_1'$ lift to the closures
  of their complements with respect to $A_1$ and hence cannot contain
  fibres of $A_1$ and hence cannot contain the lift of $B_1$ or the
  intersection. They intersection of the lifts in $[M;A_1]$ is the
  intersection of the lift of $B_2$ and the front face. No other element of
  $\mathcal{C}_1'$ can contain this, since then it would contain $B_2$
  contrary to assumption. Thus the elements of $\mathcal{C}_1'$ lifted to
  $[M;A_1]$ do not contain the intersection of the lifts of $B_1$ and $B_2$
  so after their blow up the inclusion \eqref{24.5.2008.151} still
  holds. On the other hand the elements of $\mathcal{C}_1''$ do contain the
  lift of $B_2$ and continue to do so after all elements of
  $\mathcal{C}_1'$ have been blown up. They therefore contain the
  intersection of the lifts but cannot contain the lift of $B_1.$ Again
  $\mathcal{C}_1''$ can be decomposed using Lemma~\ref{24.5.2008.148} and a
  minimal element can be blown up. For the elements which contain this
  minimal one the argument now proceed as for $A_1$ and $\mathcal{C}_1$
  above, except that the lift of $B_1$ can never be contained in these
  centres. This means that \eqref{24.5.2008.151} holds at the end of the
  blow up of one of the transversal parts of $\mathcal{C}_1''.$ However the
  other transversal components lift under these blow ups to their
  preimages, so they contain the lift of $B_2$ but not of $B_1$ and the
  argument can be repeated. Thus at the end of the blow up of
  $\mathcal{C}_1,$ \eqref{24.5.2008.151} holds. However the other
  transversal components of $\mathcal{C}$ again lift to their preimages so
  contain the intersection of the lift of $B_1$ and $B_2$ (and even the
  lift of the intersection). So the argument above for $\mathcal{C}_1$ can
  be repeated a finite number of times to finally conclude that
  \eqref{24.5.2008.151} remains true in $[M;\mathcal{C}].$
\end{proof}

\section{Boundary configuration spaces}\label{b-conf}

Let $X$ be a compact manifold with boundary and consider $M=X^n$ for some
$n\ge2.$ The boundary faces of $X^n$ are just $n$-fold products with each
factor either $X$ or a component of its boundary. Define
$\mathcal{B}_{\bo}\subset\mathcal{M}_{(2)}(X^n)$ to be equal to
$\mathcal{M}_{(2)}(X^n)$ if the boundary of $X$ is connected, otherwise to
be the proper subset consisting of those $n$-fold products where each
factor is either $X$ or the same component of the boundary in the remaining
factors, and where there are at least two of these factors.

\begin{lemma}\label{ScatProd.44} The collection
  $\mathcal{B}_{\bo}\subset\mathcal{M}(X^n)$ is closed under
  non-transversal intersection.
\end{lemma}

\begin{proof} The intersection of two elements where all boundary factors
  arise from the same boundary component of $X$ are certainly in
  $\mathcal{B}_{\bo}.$ So consider two elements $B_1,$ $B_2$ of
  $\mathcal{B}_{\bo}$ with different boundary components, $A_1\subset X$ for
  the first and $A_2\subset X$ for the second. Then $A_1\cap
  A_2=\emptyset,$ since $X$ is a manifold with boundary, so has no
  corners. Thus if different boundary components occur in any one factor in
  $B_1$ and $B_2$ then $B_1\cap B_2=\emptyset.$ The only remaining case is
  when each boundary factor in $B_1$ corresponds to a factor of $X$ in
  $B_2$ and then the intersection is transversal. Thus $\mathcal{B}_{\bo}$
  is closed under non-transversal intersection.
\end{proof}

\begin{definition}\label{ScatProd.3} The $n$-fold b-stretched
  product of $X$ is defined to be  
\begin{equation}
X^n_{\bo}=[X^n;\mathcal{B}_{\bo}].
\label{ScatProd.4}\end{equation}
\end{definition}
\noindent This definition relies on Proposition~\ref{ScatProd.1} and
Definition~\ref{ScatProd.6} to make it meaningful. Boundary faces of
codimension one, or indeed the whole of $X^n,$ could be included since
blow up of these `boundary faces' is to be interpreted as the trivial operation.

\begin{remark}\label{ScatProd.42} We will generally concentrate on the case
  that $X$ has one boundary component so \eqref{ScatProd.4} amounts to
  blowing up all the boundary faces; in this case $\cB_{\bo}=\cB_{(2)}.$
  Even if the boundary of $X$ is not connected then blowing up all elements
  of $\cB_{(2)}=\mathcal{M}_{(2)}(X^n),$ in an intersection order, is
  perfectly possible. The result may be called the `overblown' product
\begin{equation}
X^n_{\ob}=[X^n;\cB_{(2)}(X^n)],\ \pa X\text{ not connected.}
\label{ScatProd.49}\end{equation}
Since we are mainly interested in considering the resolution of
diagonals, the smaller manifold in \eqref{ScatProd.4} is more appropriate here.
\end{remark}

Next we give a more significant application of Proposition~\ref{ScatProd.1}.

\begin{proposition}\label{ScatProd.5} If $m<n,$ each of the projections off
$n-m$ factors of $X,$ $\pi:X^n\longrightarrow X^m,$ fixes a unique
  `b-stretched projection' $\pi_{\bo}$ giving a commutative diagramme
\begin{equation}
\xymatrix{
X^n_{\bo}\ar[r]^{\pi_{\bo}}\ar[d]_{\beta}& X^m_{\bo}
\ar[d]^{\beta }\\
X^n\ar[r]_{\pi}&X^m}
\label{9.1.2008.1}\end{equation}
and furthermore $\pi_{\bo}$ is a b-fibration.
\end{proposition}

\begin{proof} The existence of $\pi_{\bo}$ follows from
  Corollary~\ref{ScatProd.45}. Namely, taking $\pi$ to be the projection
  off the last $n-m$ factors for simplicity of notation, the subcollection
  of $\mathcal{B}^{\ver(\pi)}_{\bo}$ for $X^n,$ consisting of the boundary faces of
  $X^n$ in which the last $n-m$ factors consist of $X,$ is closed under
  non-transversal intersection. Thus, using Corollary~\ref{ScatProd.45},
  there is an iterated blow-down map
\begin{equation}
f:[X^n;\mathcal{B}_{\bo}]\longrightarrow [X^m;\mathcal{B}_{\bo}\times X^{n-m}].
\label{ScatProd.50}\end{equation}
Composing this with projection off the last $n-m$ factors gives a map
$\pi_{\bo}$ for which the diagramme \eqref{9.1.2008.1} commutes. Since
both the iterated blow-down map and the projection are b-submersions, so is
$\pi_{\bo}.$ To see that it is a b-fibration it suffices to show that each
boundary hypersurface of $X^n_{\bo}$ is mapped into either a boundary
hypersurface of $X^m_{\bo}$ or onto the whole manifold; this is
`b-normality'. As a b-map $\pi_{\bo}$ maps each boundary face into a
boundary face so it is enough to see what happens near the interior of each
boundary hypersurface of $X^n_{\bo}.$ If the boundary hypersurface in
question is not the result of some blow up then $\pi_{\bo}$ looks locally
the same as $\pi$ and local b-normality follows. If it is the result of
blow up then $\pi_{\bo}$ maps into the interior provided the boundary face
is not the lift of a boundary face, necessarily of codimension two or
greater, from $X^m.$ If it is such a lift then $\pi_{\bo}$ is locally the
projection onto $X^m_{\bo},$ i\@.e\@. maps into the interior of the
corresponding front face.
\end{proof}

We shall analyze more fully the structure of the boundary faces of
$[X^n;\mathcal{C}]$ where $\mathcal{C}\subset\mathcal{B}_{\bo}$ is some
collection closed under non-transversal intersection. Unless otherwise
stated below, although mostly for notational reasons, we will make the
simplifying restriction that
\begin{equation}
\text{The boundary of }X\text{ is connected so
}\mathcal{B}_{\bo}=\cB_{(2)}=\mathcal{M}_{(2)}(X^n).
\label{ScatProd.51}\end{equation}

For a boundary face $B\in\cB_{(2)}$ it is convenient to consider three
distinct possibilities
\begin{enumerate}
\item[(i)] $B\in\mathcal{C}$
\item[(ii)] $B\notin\mathcal{C}$ but there exists $A\in\mathcal{C},$ $B\subset A.$
\item[(iii)] $B\notin\mathcal{C}$ and
$A\supset B\Longrightarrow A\notin\mathcal{C}.$
\end{enumerate}

In the first case  
\begin{equation}
\mathcal{C}=\{B\}\cup\Sm(B)\cup\Bi(B)\cup\Nc(B) 
\label{ScatProd.12}\end{equation}
is a disjoint union, where 
\begin{equation}
\begin{gathered}
\Sm(B)=\{B'\in\mathcal{C};B'\subsetneqq B\}\\
\Bi(B)=\{B'\in\mathcal{C};B'\supsetneqq B\}\\
\Nc(B)=\{B'\in\mathcal{C};B\Mand B'\text{ are not comparable}\}.
\end{gathered}
\label{ScatProd.52}\end{equation}

\begin{proposition}\label{ScatProd.11} If $\mathcal{C}\subset\cB_{\bo}$ is
  closed under non-transversal intersection and $B\in\mathcal{C}$ then
  under any factor exchange map of $X^n$ which corresponds to a permutation of
  $\{1,\dots,n\}$ transforming $B$ to $(\pa X)^{c}\times X^{n-c},$
  $c=\codim(B),$ the lift of $B\in\mathcal{C}$ to $[X^n;\mathcal{C}]$ is
  diffeomorphic to
\begin{equation}
(\pa X)^{c}\times
[\bbS^{c-1,c-1};\mathcal{C}_{\Bi}]\times
[X^{n-c};\mathcal{C}_{\Sm}]
\label{ScatProd.13}\end{equation}
where $\mathcal{C}_{\Bi}$ is the collection of boundary faces
of the totally positive part $\bbS^{c-1,c-1}$ of the
$(c-1)$-sphere corresponding to the elements of
$\Bi(B)$ in \eqref{ScatProd.12} and
$\mathcal{C}_{\Sm}$ is the collection of boundary faces
$B'\subset X^{n-c}$ arising from the elements of $\Sm(B)$ in
\eqref{ScatProd.12}.
\end{proposition}

\begin{proof} The ordering of $\mathcal{C}$ arising from
  \eqref{ScatProd.12}, in which the pieces are size-ordered, is an
  intersection-order. Since $\mathcal{C}$ is closed under non-transversal
  intersection, each element $B'\in\Nc(B)$ corresponds to an element
  $B'\cap B\in\Sm(B).$ Once this is blown up the
  lifts of $B$ and $B'$ are disjoint, which accounts for the
  absence of terms from $\Nc(B)$ in \eqref{ScatProd.13}. Moreover
  this argument shows that the result as far as $B$ is concerned is
  the same if $\mathcal{C}$ is replaced by the union of the first three
  terms in \eqref{ScatProd.12}. Relabelling the factors so that $B$ has
  the boundary of $X$ in the first $c$ factors, the result of blowing up
  $B$ in $X^n$ is to replace a neighbourhood of it by 
\begin{equation}
(\pa X)^{c}\times \bbS^{c-1,c-1}\times X^{n-c}\times[0,1)
\label{ScatProd.14}\end{equation}
where the last factor is a defining function for the front face. Moreover,
the boundary faces in $\mathcal{C}$ (excluding those not comparable to
$B)$ lift either to the products of boundary faces of
$\bbS^{c-1,c-1}$ with the other factors except the last,
or else products of all the other factors with boundary faces of
$X^{n-c}.$ The effect of the subsequent blow-ups on the lift of
$B$ is therefore as indicated in \eqref{ScatProd.13}. 
\end{proof}

So, next suppose instead that $B\notin\mathcal{C}.$ The
decomposition \eqref{ScatProd.12} still exists, of course without $B$
itself. Case (ii) above corresponds to $\Bi(B)$ being non-empty. Since
$\Bi(B)\subset\mathcal{C}$ consists of those elements which contain $B$ it
is also closed under non-transversal intersection. For a fixed element
$A'\in\Bi(B)$ no two elements contained in $A'$ can be transversal, so this
subcollection is closed under intersection and hence has a minimal
element. Since $\Bi(B)$ is closed under non-transversal intersection, the
collection of minimal elements must be transversal in pairs. Denote this
collection 
\begin{equation}
\begin{gathered}
\mathfrak{b}(B)=\{A\in\mathcal{C};B\subset A,\
B\subset A'\subset A,\ A'\in\mathcal{C}\Longrightarrow A'=A\}\Mthen\\
\Bi(B)=\bigcup_{A\in\mathfrak{b}(B)}\Bi_{A}(B)\text{ is a disjoint
  union, where }\Bi_A(B)=\{A'\in\mathcal{C}; A\subset A'\}.
\end{gathered}
\label{ScatProd.79}\end{equation}

\begin{proposition}\label{ScatProd.15} If $B\notin\mathcal{C},$ where
  $\mathcal{C}\subset\cB_{(2)}$ is closed under
  non-transversal intersection, but $\Bi(B)\not=\emptyset$ then
  $\mathfrak{b}(B)\subset\mathcal{C}$ defined by \eqref{ScatProd.79} is a
  non-empty collection of transversally intersecting boundary faces and $B$ lifts
  to be a common boundary face (i\@.e\@. in the intersection of) the lifts of
  the elements of $\mathfrak{b}(B),$ in \eqref{ScatProd.79}, and is
  diffeomorphic to
\begin{equation}
[B;\Sm(B)]\times
  \prod_{A_i\in\mathfrak{b}(B)}
[\bbS^{d(i)-1,d(i)-1};\Bi(A_i)],\ d(i)=\codim(A_i), 
\label{ScatProd.17}\end{equation}
where $\Sm(B)$ is interpreted as a collection of boundary faces of $B$ 
and where for each $A_i\in\mathfrak{b}(B),$ $\Bi(A_i)$ is the collection of
boundary faces of $\bbS^{d(i)-1,d(i)-1}$ arising from the lifts of the
elements of $\Bi(B)$ strictly containing $A_i.$ 
\end{proposition}

\begin{proof} Give $\mathcal{C}$ an intersection-order in which the
  elements of $\mathfrak{b}(B)$ come first, followed by the other elements
  of $\Bi(B)$ in a size-order, followed by the elements of $\Sm(B)$ in
  size-order, followed by the elements of $\Nc(B),$ also in
  size-order. Since the elements of $\mathfrak{b}(B)$ are transversal, they
  can be in any order. Since $B$ lifts into the front face under the first
  blow-up and remains a boundary face of the lifts of the others it lifts
  to be in the intersection of the lifts of the elements of
  $\mathfrak{b}(B)$ and after they are blown up is of the form 
\begin{equation}
B\times \prod_{A_i\in\mathfrak{b}(B)}
\bbS^{d(i)-1,d(i)-1},\ d(i)=\codim(A_i).
\label{ScatProd.80}\end{equation}
The other elements of $\Bi(B)$ contain one of the $A_i$ and are transversal
to the others and it follows that they lift to be boundary faces of the
corresponding fractional sphere, as indicated. The boundary faces in
$\Sm(B)$ lift in the obvious way and \eqref{ScatProd.17} results from the
fact that the subsequent blow ups of boundary faces not comparable to $B$
do not affect its lift, since their intersections with $B$ have already
been blown up.
\end{proof}

The third case is then $B\notin\mathcal{C}$ and such that there is no
element of $\mathcal{C}$ containing it.

\begin{proposition}\label{ScatProd.16} If $B\subset A$ implies that
  $A\notin\mathcal{C}$ then $B$ lifts to a boundary face of
  $[X^n;\mathcal{C}]$ of the same dimension which is diffeomorphic to 
\begin{equation}
[B;\{F\in\mathcal{M}_{(2)}(B);F=G\cap B,\ G\in\mathcal{C}\}].
\label{24.5.2008.117}\end{equation}
\end{proposition}

\begin{proof} Give $\mathcal{C}$ the intersection order in which all the
  $B'\subset B$ come first (size-ordered) and then all the faces
  which are not comparable to $B$ (size-ordered as well). The effect
  on $B$ is then as indicated!
\end{proof}

\section{Multi-diagonals}\label{Multidiags}

The main utility of the manifold $X^n_{\bo}$ as constructed above is that
it resolves the intersection with the boundary of each of the
multi-diagonals in $X^n.$ The total diagonal in $X^n$ is the submanifold,
diffeomorphic to $X,$ which is the image of the map $X\ni p\longmapsto
(p,p,p,\dots,p)\in X^n,$
\begin{equation}
\Diag(X^n)=\{m\in X^n;m=(p,p,\dots, p)\text{ for some }p\in X\}.
\label{ScatProd.76}\end{equation}
The partial diagonals in $X^n$ are the inverse images of the total
diagonal in $X^k,$ for some $k\le n,$ pulled back to $X^n$ by one of the
projections $X^n\longrightarrow X^k.$ Thus a partial diagonal involves
equality in at least two factors. Since we are assuming that the boundary
of $X$ is connected, there is a 1-1 correspondence between partial
diagonals, projections onto $X^k$ for $k\ge2,$ and elements of
$\cB_{\bo}.$ Thus if $B\in\cB_{\bo}$ it will be
convenient to denote the correspond partial diagonal as $D_B$ and the
corresponding projection as $\pi_B,$ where the partial diagonal is given by
equality in exactly those factors in which $B$ has a boundary component.

Diagonals are however more complicated than boundary faces, at least when
$n\ge4.$ Namely the intersection of two partial diagonals is not
necessarily a partial diagonal. Indeed
\begin{equation}
D_B\cap D_{B'}=D_{B\cap B'}\text{ if and only if }B\cap B'\text{ is
  non-transversal}. 
\label{ScatProd.72}\end{equation}
Non-transversality of intersection of $B$ and $B'$ is the condition that
$\pa X$ appears in at least one factor in common. Then equality of points
in all the factors in which the boundary occurs in $B$ and also in $B'$
separately, implies that they are equal in all factors in $B\cap B'$ giving
equality in \eqref{ScatProd.72} in the non-transversal case. Conversely, if
the intersection is transversal then separate equality does not imply
overall equality. Thus

\begin{lemma}\label{ScatProd.73} The submanifolds of $X^n$ arising
as the intersections of collections of partial diagonals in $X^n$ form the
collection of \emph{multi-diagonals} which are in 1-1 correspondence with
the (non-empty) transversal collections of boundary faces of codimension at
least two
\begin{equation}
\mathfrak{b}\subset\cB_{(2)}\Mst B_1,B_2\in\mathfrak{b}\Longrightarrow
B_1\pitchfork B_2.
\label{ScatProd.74}\end{equation}
\end{lemma}

The multi-diagonal corresponding to the transversal family $\fb$ will
be denoted by $D_{\fb}$, so
\begin{equation}
D_{\mathfrak{b}}=\{m\in
X^n;\pi_{B}(m)\in\Diag(X^k),\ k=\codim(B)\ \forall\ B\in\mathfrak{b}\}.
\label{ScatProd.75}\end{equation}

Let us now clarify the sense in which $X^n_{\bo}$ resolves the
multi-diagonals. Consider first the total diagonal.

\begin{lemma}\label{ScatProd.53} The total diagonal in $X^n$ is naturally
  diffeomorphic to $X$ and is a b-submanifold but never a
  p-submanifold. Under blow-up of the boundary face of maximal codimension,
  $(\pa X)^n,$ the total diagonal lifts to (i\@.e\@. the closure of its
  interior is) a p-submanifold.
\end{lemma}

\begin{proof} A neighbourhood of this `maximal corner' of $X^n$ is of the
  form $[0,1)^n\times (\pa X)^n$ with the coordinate in each of the first
    factors given by a fixed boundary defining function lifted from each
    factor and then denoted $x_j.$ The total diagonal meets this in the
    b-submanifold $\{x_1=\dots=x_n\}\times\Diag_{\tot}(\pa X)^n.$ After
    blowing up the corner a neighbourhood of the front face is of the form
\begin{equation}
[0,1)_\rho\times\bbS^{n-1,n-1}\times(\pa X)^n,\ \rho =x_1+\dots+x_n
\label{ScatProd.54}\end{equation}
to which the total diagonal lifts as
\begin{equation}
[0,1)\times\{\bar\omega\}\times\Diag_{\tot}(\pa X)^n,\ \bar\omega\in\bbS^{n-1,n-1}
\label{ScatProd.55}\end{equation}
being the `centre' of the fractional sphere and hence an interior
point. This shows that the total diagonal in $X^n$ is resolved to a
p-submanifold.
\end{proof}

The general case of a multi-diagonal $D_{\mathfrak{b}}$ is similar;
the next lemma shows that it is resolved in $[X^n;\cC]$ for any
collection of boundary faces $\cC,$ closed under non-transversal
intersection and containing $\fb.$

\begin{proposition}\label{ScatProd.7}  Let $\cC$ be an
  intersection-ordered family of boundary faces of $X^n$ and $\fb\subset
  \cC$ a transversally intersecting subcollection then the lift of $D_\fb$
  to $[X^n;\cC]$ is a p-submanifold.
\end{proposition}


\begin{proof} Let $\cC = \fb \cup \cC'$ (disjoint union).  Because
  $\cC$ is closed under non-transversal intersection and $\fb$ is a
  transversal collection, there is an intersection order on $\cC$ such that
  all elements of $\fb$ come before any element of $\cC'$, and $\cC'$
  itself is size-ordered. We claim that the lift of $D_\fb$ to 
  $[X^n;\fb]$ is a p-submanifold.

If a factor-exchange map is used to identify
\begin{equation}
X^n \equiv X^{n-k}\times \prod_{i=1}^LX^{k_i}
\label{ScatProd.77}\end{equation}
in such a way that the each element $B_i\in\mathfrak{b}$ is
identified with the corner of maximal codimension in $X^{k_i}$ then the
space obtained by blow up of the elements of $\mathfrak{b}$ is identified
smoothly with
\begin{equation}
X^{n-k}\times\prod_{i=1}^L[X^{k_i};(\pa X)^{k_i}].
\label{ScatProd.78}\end{equation}
Then Lemma~\ref{ScatProd.53} shows that $D_{\mathfrak{b}},$ which is
identified with the product of $X^{n-k}$ and the maximal diagonals in the
$X^{k_i},$ is resolved to a p-submanifold in $[X^n;\mathfrak{b}].$ 

As noted earlier, under the blow-up of boundary faces, a p-submanifold (in
this case an interior p-submanifold) lifts to a p-submanifold. This proves
that the lift remains a p-submanifold under subsequent blow-up of the
elements of $\cC'$.
\end{proof}

Let us next gather some notation and information about intersections of
multi-diagonals. Since the intersection of any two multi-diagonals is
another multi-diagonal, given $\fb_1$ and $\fb_2,$ there exists a
transversal family $\fb_1\Cup\fb_2$ uniquely defined by the condition
\begin{equation}
D_{\mathfrak{b_1}}\cap D_{\mathfrak{b_2}}=D_{\mathfrak{b_1}\Cup\mathfrak{b_2}}.
\label{ScatProd.82}\end{equation}
The family $\fb_1\Cup \fb_2$ will be called the `transversal union' of
$\fb_1$ and $\fb_2.$ It is defined as follows. Partition
$\mathfrak{b_1}\cup\mathfrak{b_2}$ into subsets where two elements lie in
the same subset if and only if there is a chain of elements connecting
them, each intersecting the next non-transversally. Then the elements of
$\fb_1\Cup\fb_2$ consist of the intersections over these subsets.

Clearly, if all pairs $(B_1,B_2)\in \fb_1\times \fb_2$ meet transversally
then $\fb_1\Cup\fb_2 = \fb_1\cup \fb_2.$ Otherwise the transversal union
has fewer elements than the union and they need not be elements of either
collection.

It is also convenient to introduce the following notation:
\begin{itemize}
\item Write $\mathfrak{b}_1\pitchfork\mathfrak{b}_2$ if
  $D_{\fb_1}\pitchfork D_{\fb_2},$ or equivalently if $\fb_1\cup \fb_2
  = \fb_1\Cup \fb_2.$
\item Say $\mathfrak{b}_1$ and $\mathfrak{b}_2$ are comparable if
$\fb_1\Cup\fb_2=\fb_2$ or $\fb_1\Cup\fb_2=\fb_1$ which is equivalent to
  $D_{\fb_2}\subset D_{\fb_1}$ or $D_{\fb_1}\subset D_{\fb_2}.$
\item Otherwise say $\fb_1$ and $\fb_2$ are \ncnt: this is equivalent
to $D_{\fb_1}$ and $D_{\fb_2}$ being \ncnt, or combinatorially, to
the condition that $\mathfrak{b}_1\Cup\mathfrak{b}_2$ is neither
the union nor either of the individual sets of boundary faces.
\end{itemize}

To motivate the discussion of the next section, let us give a local
coordinate description of the multi-diagonal $D_\fb$ and its lift to
$[X^n;\fb].$ Explicitly, there is a set $(I_1,\ldots,I_L)$ of
disjoint subsets of $\{1,\ldots,n\}$, each of cardinality $\geq 2$,
such that
\begin{equation}\label{e1.21.7.8}
  D_\fb = \cap_{r=1}^L\{z_k=z_l\mbox{ for all }k,l\in I_r\}.
\end{equation}
Using adapted local coordinates $z = (x,y)$ near the boundary of $X,$ a
full set of local boundary defining functions for $[X^n;\fb]$ are given by
the $t_{r},$ $B_r\in\fb,$ which are the sums 
\begin{equation}
T_{r}=\sum\limits_{j\in I_r}x_j
\label{24.5.2008.160}\end{equation}
of the local defining functions for $B_r$ and the $t_j=x_j/T_{r}$ if
$j\in B_r$ for some $r$ and and $t_j=x_j$ otherwise. Interior coordinates
lift to interior coordinates. 

In order to describe the lift of $D_\fb$ to $[M;\fb]$ introduce new
variables  
\begin{equation}\label{e3.21.7.8}
s_j = \log t_j,\; u_j = y_j - y_{a_r} \mbox{ if }j \in J_r\mbox{ for
  some }r.
\end{equation}
Then the variables
\begin{equation}
(s_j,u_j,t_k,y_k)\mbox{ for }j \not\in J,\; k \in J,\mbox{ where
}t_k\geq 0
\end{equation}
($J = J_1\cup\ldots\cup J_L$) form an adapted local coordinate system
on $[X^n;\fb]$ with respect to which
\begin{equation}\label{e4.21.7.8}
\tilde{D}_\fb = \{s_j=0, u_j=0;j\in J\}.
\end{equation}

Let us now consider the family $\cD(\fb)$ of all multi-diagonals containing
$D_\fb.$ It is clear that if $D \in \cD(\fb),$ then $D$ must be an
intersection of the form
\begin{equation}
D =\{s_j=0, u_j=0;j \in K_0\}
\cap\cap_{r=1}^M\{s_i=s_j, u_i=u_j\mbox{ for all }i,j\in K_r\}.
\end{equation}
where $K_0,K_1,\ldots, K_r$ are disjoint subsets. In fact, for each
$r=1,\ldots, M$, $K_r$ must be contained in one of $J_1,\ldots,J_L.$

This discussion shows that any given multi-diagonal lifts to a
p-submanifold, but that there is no single system of adapted coordinates
which put all elements of $\cD$ in standard form.  The notion of a
d-collection, which we introduce in the next section, is designed to
capture the local structure of families like $\cD.$

\section{D-collections}\label{D-coll}

Next we introduce a notation for collections of p-submanifolds which
includes the resolutions of diagonals.

Any p-submanifold is locally of the form \eqref{24.5.2008.161} in adapted
coordinates. The \emph{interior codimension} is $d=|I|;$ $Y$ is an interior
p-submanifold if $k=0,$ i\@.e\@.~no boundary variables are involved in its
definition, otherwise it is contained in a unique boundary face of
codimension $k$ (its \emph{boundary hull} B(Y)). If $N$ is such that $d+N$
is less than or equal to the number of interior variables then $Y$ can
alternatively be brought to a {\em local diagonal} form relative to the
adapted coordinates in the sense that
\begin{equation}
U\cap Y = \{(x,y) \in U; x_i=0,\ 1\le i\le k,\ y_i = y_j\text{ for all }i,j\in
\mathfrak{l}_l,\ l=1,\dots,N\}
\label{24.5.2008.157}\end{equation}
where the $\mathfrak{l}_l$ are some disjoint subsets (including possibly
none) each having cardinality $\geq 2.$

Of course to get \eqref{24.5.2008.157} one just needs to divide the
interior coordinates into groups and subtract one of $N$ of the remaining
interior variables from each element of each set. The interior
codimension of $Y$ (which is constant, since $Y$ is connected by
assumption) is $\sum\limits_{l}(|\mathfrak{l}_l|-1).$

The important property of boundary diagonals that we wish to capture in the
notion of a d-collection is that they can simultaneously brought to such
diagonal form near any point.

\begin{definition} A collection $\cE$ of p-submanifolds (if not connected
  then each must have fixed dimension) in $M$ is called a
  \emph{d-collection} if for each point $p\in M$ there is one set of adapted
  coordinates based at that point in terms of which all the elements of
  $\cE$ through that point take the form \eqref{24.5.2008.157}.
\end{definition}
\noindent Clearly this condition is void locally at any point not contained
in one of the elements of $\cE$ and for any single p-submanifold which does
not have maximal interior codimension. Any collection of boundary faces can
be added to a d-collection and it will remain a d-collection since they are
automatically of the form \eqref{24.5.2008.157} (for any adapted
coordinates) for the empty collection of disjoint sets $\mathfrak{l}_l.$

We will decompose a d-collection into the subcollections of elements which
are and those which are not boundary faces.
\begin{equation}
\cE=\cE_{\bo}\cup\cE',\ \cE_{\bo}=\cE\cap\mathcal{M}(M).
\label{24.5.2008.149}\end{equation}
As usual, including a boundary hypersurface in a given collection $\cE$ is
a matter of convention; for the sake of definiteness we exclude hyerpsurfaces.

\begin{lemma}\label{24.5.2008.158} Any subcollection of the boundary faces
  of the elements of a d-collection (with the addition of any of the
non-hypersurface boundary faces of the manifold) is a d-collection.
\end{lemma}

\begin{proof} Immediate from the definition.
\end{proof}

As already mentioned, the lifts of the diagonals give examples of
d-collections provided the appropriate boundary faces have been blown up.

\begin{proposition}\label{ScatProd.56} If $\mathcal{C}\subset\cB_{\bo}$ is
  closed under non-transversal intersection then all the diagonals
  $D_{\mathfrak{b}}$ with $\mathfrak{b}\subset \mathcal{C}$
  lift from $X^n$ to $[X^n;\mathcal{C}]$ to interior p-submanifolds which form a
  d-collection.
\end{proposition}

\begin{proof} This follows from the discussion at the end of the preceding
section since the same coordinates work for all diagonals.
\end{proof}

For a d-collection $\cE$ of p-submanifolds we consider a closure
condition corresponding to the index sets in \eqref{24.5.2008.157} that
define them.  Let $\mathfrak{l}$ and $\mathfrak{l}'$ be two subpartitions
of the index set (of interior coordinates). Then as in \S\ref{Sect-Md} we write
\begin{equation}
\begin{gathered}
\mathfrak{l}\Subset\mathfrak{l}'\text{ if each set
}\mathfrak{l}_i\in\mathfrak{l}\text{ is contained in one of the 
}\mathfrak{l}'_j\\ 
\mathfrak{l}\pitchfork\mathfrak{l}'\text{ all sets
}\mathfrak{l}_i\text{ are disjoint from all sets }\mathfrak{l}_j.
\end{gathered}
\label{24.5.2008.140}\end{equation}
Thus in the second case $\mathfrak{l}\cup\mathfrak{l}'$ is still a
subpartition. 

Now the condition we impose on $\cE$ concerns the elements which are not
boundary faces, and which pass through a given point
\begin{equation}
\begin{gathered}
\forall\ E,\ E'\in\cE'\Mand p\in E\cap E'\text{ if neither condition
  in \eqref{24.5.2008.140} holds}\\
\text{for the index sets defining them then}\\
\exists\ F\in \cE',\ p\in F\text{ with the same index set as }E\cap
E',\\
\text{containing it and with boundary hull in }B(E)\dotplus B(E').
\end{gathered}
\label{24.5.2008.139}\end{equation}

\begin{proposition}\label{24.5.2008.137} If $\cE$ is a d-collection of
p-submanifolds all contained in (or equal to) proper boundary faces of a
manifold with corners $M$ for which the closure condition
\eqref{24.5.2008.139} holds then on the blow up of an
element $G\in\cE_{\bo}$ or an element $E\in\cE'$ of maximal interior codimension
the elements of $\cE\setminus\{E\}$ lift to a d-collection in $[M;E]$ which
again satisfies the closure condition.
\end{proposition}

\begin{proof} Certainly the blow up of $E$ is well-defined, since it is a
  p-submanifold by assumption.

The simplest case is if $G\in\cE_{\bo}$ is actually a boundary face. Then all
the other p-submanifolds certainly lift to p-submanifolds. The other
elements of $\cE_{\bo}$ lift to boundary faces, the elements of $\cE'$
which are not contained in $G$ lift to the closures of the complements with
respect to $G$ and the elements of $\cE'$ which are contained in $G$ lift
to their preimages. Away from $G$ nothing has changed and near it, the
boundary defining functions $x_1,\dots,x_k$ which define it are replace by
their sum $T_G$ and $t_j=x_j/T_G.$ The defining conditions involving
interior variables are unchanged by the blow up. The changes of
intersections of elements of $\cE'$ correspond to whether they are
contained in $G$ or not and so \eqref{24.5.2008.157} persists everywhere
locally, with only the boundary functions changing. The closure condition
also persists at every point of intersection after blow up, since the only
problem would from $E,$ $E'\subset G$ but $F\setminus G\not=\emptyset,$
since then the lift of $F$ would not contain the lift of the intersection
of $E$ and $E'.$ The last condition in \eqref{24.5.2008.139}, on the hulls,
prevents this from happening, since if $E$ $E'\subset G$ both then the
boundary hull of $F$ must also be contained in $G$ and the boundary hull of
its lift must be contained in the hull of the lifts of $E$ and $E'.$

So next consider the blow up of an element $Y\in\cE'$ with boundary hull
$B.$ The condition of maximality of its interior codimension, i\@.e\@. the number
of equations defining it within $B,$ means, by \eqref{24.5.2008.139} that
the only other elements of $\cE'$ it meets must satisfy one of the
conditions in \eqref{24.5.2008.140}, since otherwise the $F$ whose
existence is demanded by \eqref{24.5.2008.139} would have larger interior
codimension. In particular the only way another element of $\cE'$ can be
contained in $Y$ is if it is a boundary face of $Y,$ hence has the same
index set but boundary hull which is a boundary face of the boundary hull
of $Y.$ On blow up of $Y$ these p-submanifolds lift to boundary faces of
the front face produced by the blow up (which is one good reason boundary
faces are allowed in the definition of d-collections). So consider elements
of $\cE'$ which meet, but are not contained in $Y.$ By the maximality of
the interior codimension of $Y,$ these corresponds to index sets in one of
the two cases in \eqref{24.5.2008.140} so fall into two classes, those with
interior defining conditions implied by the defining conditions for $Y$ and
those involving variables which are completely independent of those
defining $Y.$ The latter clearly lift to have the same defining conditions
and with hull simply the lift of the previous hull.

The blow up of $Y$ can be made explicit locally by choosing one of the
elements labelled by the $\mathfrak{l}_{i}$ and subtracting it from the
others. This changes the defining conditions for $Y$ into the vanishing of
interior variables and boundary variables, so locally the blow up
corresponds to polar coordinates in these variables. All the elements of
$\cE'$ meeting $Y,$ but not contained in it, and corresponding to the first
case in \eqref{24.5.2008.140} have lifts defined by some of the new
boundary variables (not including $\rho _{\ff})$ and the vanishing of some
of the interior polar variables. It follows that the same intersection
property \eqref{24.5.2008.139} results by simply adding the lifted `extra
variable' to each of these lifted interior variables. Thus it follows that
the the d-submanifold condition holds for the lift of the elements of
$\cE\setminus\{Y\}.$

It remains only to check the closure condition, but this persists from the
same arguments since the local index sets have not changed.
\end{proof}

Notice that the interior codimension of each element of the lift of the
d-collection on the blow up of a boundary face is the same as before blow
up. On blow up of an element of $Y\in\cE'$ with maximal interior codimension,
all the lifts have the same interior codimension except for any which are
boundary faces of $E$ itself, which lift to boundary faces of $[M;E]$ and
hence are subsequently boundary faces, for which the notion of an interior
codimension is not defined.

\section{Boundary diagonals}\label{b-diag}

In Proposition~\ref{ScatProd.7} it is shown that the diagonal
$D_{\mathfrak{b}}$ associated to a transversal subset
$\mathfrak{b}\subset\cB_{(2)}$ lifts to a p-submanifold 
of $[X^n;\mathcal{C}]$ provided $\mathfrak{b}\subset\mathcal{C}.$ If
$A\in\cM(\cap\fb)$ is a boundary face of the intersection of the elements
of $\fb$ and $\tilde A$ is its lift to $[X^n;\fb]$ we denote the
intersection by
\begin{equation}
H_{A,\mathfrak{b}}=\tilde A\cap D_{\fb}\ A\subset\cap\mathfrak{b}.
\label{ScatProd.58}\end{equation}
These are all p-submanifolds, indeed they are each interior p-submanifolds of the
corresponding boundary face $\tilde A,$ henceforth denoted $A$ again, which
is the boundary-hull, i\@.e\@. $H_{A,\mathfrak{b}}$ is contained in no
smaller boundary face than $A.$ As such their lifts are always well-defined
under blow up of boundary faces (for us only lifted from $X^n)$ and
$H_{A,\mathfrak{b}}$ remains an interior p-submanifolds of the lift of $A.$
Thus we conclude that
\begin{equation}
\begin{gathered}
\text{Provided }\mathfrak{b}\subset\mathcal{C},\
H_{A,\mathfrak{b}}\subset[X^n;\mathcal{C}]
\text{ is an interior p-submanifold of }\\
A\text{ lifted to }[X^n;\mathcal{C}],\ \forall\ A\subset
B(\fb)=\cap\mathfrak{b}.
\end{gathered}
\label{ScatProd.59}\end{equation}

If fact if $\mathfrak{b}\subset\mathcal{C}$ then \eqref{ScatProd.58} still
holds in after further blow ups:
\begin{equation}
H_{A,\mathfrak{b}}=A\cap D_{\mathfrak{b}}
\Min[X^n;\mathcal{C}],\ A\subset\cap\mathfrak{b}.
\label{ScatProd.62}\end{equation}
This allows us to compute the intersection of $H_{A,\mathfrak{b}}$ and
$H_{A',\mathfrak{b}'}$ in any $[X^n;\mathcal{C}]$ in which they are both
defined, i\@.e\@. if $\mathfrak{b}\cup\mathfrak{b}'\subset\mathcal{C},$
$A\subset\cap\mathfrak{b}$ and $A'\subset\cap\mathfrak{b}'.$

First, if $\mathfrak{b},$ $\mathfrak{b}'\subset\cB_{(2)}$ are each
transversal subsets their `transversal union' is defined by, and following,
\eqref{ScatProd.82}.

Consider the boundary diagonals which lie in a given boundary face. As
already noted, if $A\in\cB_{(2)}$ then $H_{A,\mathfrak{b}}$ is a
p-submanifold of the lift of $A$ to $[X^n;\mathcal{C}]$ provided
$\mathfrak{b}\subset\mathcal{C}$ and $A\subset\cap\mathfrak{b}.$ We need to
blow up all these submanifolds, as $A$ and $\mathfrak{b}$ vary over all
such possibilities. Notice that if $\mathfrak{b},$
$\mathfrak{b}'\subset\mathcal{C},$ with the latter closed under
non-transversal intersection, then
$\mathfrak{b}\Cup\mathfrak{b}'\subset\mathcal{C}$ since the elements are
all non-transversal intersections of elements of $\mathfrak{b}$ and
$\mathfrak{b}'.$

From \eqref{ScatProd.82} we conclude:

\begin{lemma}\label{ScatProd.83} If
  $\mathfrak{b}_i\subset\mathcal{C}\subset\cB_{(2)}$ are two 
  transversal subsets for $i=1,2$ and
  $A\subset(\cap\mathfrak{b}_1)\cap(\cap\mathfrak{b}_2)$ is a common
  boundary face then $\mathfrak{b}_1\Cup\mathfrak{b}_2\subset\mathcal{C},$
  $A\subset\cap(\mathfrak{b}_1\Cup\mathfrak{b}_2)$ and
\begin{equation}
H_{A,\mathfrak{b}_1}\cap H_{A,\mathfrak{b}_2}=H_{A,\mathfrak{b}_1\Cup\mathfrak{b}_2}\Min
[X^n;\mathcal{C}].
\label{ScatProd.84}\end{equation}
\end{lemma}

\begin{proof} This follows from the fact that we can identify $D_{\mathfrak{b}_1}$
in $[X^n;\mathfrak{b}_1],$ then blow up the elements of
  $\mathfrak{b}_1\Cup \mathfrak{b}_2$ and then $\mathfrak{b}_2.$ There is
an intersection order of $\mathcal{C}$ in which these are the first
blow-ups and are in this order; it then follows that \eqref{ScatProd.84}
holds in general. 
 \end{proof}

\begin{definition}\label{ScatProd.85} The codimension of a transversal
  collection $\mathfrak{b}\subset\cB_{(2)}$ is the sum of the
  codimensions of its elements, so is the codimension of their
  intersection. A size-order on such transversal collections is an order in
  which this total codimension is weakly decreasing. 
\end{definition}

In the second stage of the construction of the scattering n-fold product we
need to blow up all of the $H_{A,\mathfrak{b}}.$ Since this has to be done
step by step we consider a closure condition under intersection on a
collection of the submanifolds which is enough to allow them all to be
blown up unambiguously.

\begin{definition}\label{ScatProd.87} A collection
  $\cG\subset\{H_{A,\mathfrak{b}}\}$ in $[X^n;\mathcal{C}],$ so by
  assumption $H_{A,\mathfrak{b}}\in\cG$ implies
  $\mathfrak{b}\subset\mathcal{C},$ is \emph{intersection-closed} if 
\begin{equation}
H_{A_i,\mathfrak{b}_i}\in\cG,\ i=1,2\Longrightarrow
H_{A,\mathfrak{b}}\in\cG,\ A=A_1\cap
A_2,\ \mathfrak{b}=\mathfrak{b}_1\Cup\mathfrak{b}_2.
\label{ScatProd.88}\end{equation}
A \emph{chain-order} on $\cG$ is an order in which each $\mathfrak{b}$
which occurs does so only in an uninterrupted interval with the codimension of
$\mathfrak{b}$ weakly decreasing overall and when $\mathfrak{b}$ is
unchanging, the codimension of $A$ is weakly decreasing.
\end{definition}
\noindent Thus this is a `lexicographic order' in which $\mathfrak{b}$ is
the first `letter' and $A$ the second.

\begin{proposition}\label{ScatProd.86} An intersection-closed
  collection, $\cG,$ of boundary diagonals in $[X^n;\mathcal{C}]$ can be
  blown up in any chain-order (so under such blow ups all later elements lift to
  p-submanifolds) and the resulting manifold is independent of the
  chain-order chosen.
\end{proposition}

\begin{proof} It follows from Proposition~\ref{ScatProd.56} that the
  elements of $\cG$ form a d-collection of p-submanifolds in
  $[X^n;\mathcal{C}].$ To apply Proposition~\ref{24.5.2008.137} we need to
  check the closure condition \eqref{24.5.2008.139}. Consider two elements
  $H_{A_i,\mathfrak{b}_i}$ of $\mathcal{G};$ by assumption
  $H_{A,\mathfrak{b}}$ in \eqref{ScatProd.88} is also an element of
  $\mathcal{G}.$ Applying Lemma~{24.5.2008.150} to $A_1$ and $A_2$ shows
  that 
\begin{equation}
A_1\cap A_2\subset A\Min[X^n;\mathcal{C}].
\label{24.5.2008.152}\end{equation}
However from \eqref{ScatProd.62} we conclude that 
\begin{equation}
H_{A_1,\mathfrak{b}_1}\cap H_{A_2,\mathfrak{b}_2}\subset
H_{A,\mathfrak{b}}\Min [X^n;\mathcal{C}].
\label{24.5.2008.153}\end{equation}
Since the index sets as a d-collection are just the $\mathfrak{b}$'s this
shows that the closure condition \eqref{24.5.2008.139} for $\mathcal{G}$
follows from \eqref{ScatProd.88}.

Thus Proposition~\ref{24.5.2008.137} shows that the elements of
$\mathcal{G}$ can be blown up in any order so when blown up each element
has maximal interior codimension or is a boundary face. Clearly a
chain-order as defined above has this property.

To complete the proof of the Proposition it remains to show that different
chain-orders lead to the same blown up manifold. To see this means first
showing that two neighbouring elements $H_{A_i,\mathfrak{b}}$ with the same
$\mathfrak{b}$ and with $A_i$ of the same original codimension in $X^n$ can
be interchanged. By \eqref{ScatProd.88},
$H_{A_1\cap A_2,\mathfrak{b}}\in\mathcal{G}$ must already have been blown up. As
follows from \eqref{24.5.2008.153}, before it is blown up this contains the
intersection of the $H_{A_i,\mathfrak{b}}$ but cannot contain either of
them. It follows as in the case of boundary faces that after this boundary
diagonal has been blown up these two are disjoint and hence can be
interchanged. It is also necessary see that the ordering amongst the
$\mathfrak{b}$ can be changed, subject to the decrease of codimension of
$\mathfrak{b}.$ That this is possible follows from the next result which
completes the proof of the Proposition.
\end{proof}

\begin{definition}\label{24.5.2008.120} 
Let $\mathcal{H}_{*,\mathcal{C}}$ be the collection of the
$H_{A,\mathfrak{b}}\subset[X^n;\mathcal{C}]$ where
$\mathfrak{b}\subset\mathcal{C}$ is a transversal subcollection of boundary
faces and $A\subset\cap\mathfrak{b}.$ A collection
$\mathcal{G}\subset\mathcal{H}_{*,\mathcal{C}}$ is \emph{face--closed} if
$H_{A,\mathfrak{b}_1\Cup  \mathfrak{b}_2}\in\mathcal{G}$ whenever
$H_{A,\mathfrak{b}_i}\in\mathcal{G}$ for $i=1,2.$ Thus each of the subcollections
which have a given boundary face as boundary hull is closed under intersection.

Such a collection is \emph{chain--closed} if
$H_{A,\mathfrak{b}}\in\mathcal{G}\subset\cH_{*,\mathcal{C}}$ and $A'\subset 
A\subset\cap\mathfrak{b},$ implies $H_{A',\mathfrak{b}}\in\mathcal{G}.$

A collection $\mathcal{G}\subset\mathcal{H}_{*,\mathcal{C}}$ is \emph{fc-closed} if
it is both face-closed and chain-closed.
\end{definition}

Now, for a collection $\mathcal{G}\subset\mathcal{H}_{*,\mathcal{C}}$ let
$\beta(\mathcal{G})$ be the collection of transversal boundary faces which
  occurs, that is, $\mathfrak{b}\in\beta (\mathcal{G})$ if and only if
  $H_{A,\mathfrak{b}}\in \mathcal{G}$ for some $A.$

\begin{lemma}\label{24.5.2008.115} If $\mathcal{C}\subset\cB_{(2)}$ is
  closed under non-transversal intersection then
  $\mathcal{G}\subset\mathcal{H}_{*,\mathcal{B}}$ is fc-closed if and only if
it is intersection-closed in the sense of \eqref{ScatProd.88} and
\begin{equation}
\begin{aligned}
\mathfrak{b}\in&\beta(\mathcal{G})\Longrightarrow\\
&\{A\in\cB_{(2)};H_{A,\mathfrak{b}}\in\mathcal{G}\}\text{ is
  closed under passage to boundary faces.}
\end{aligned}
\label{ScatProd.91}\end{equation}
\end{lemma}

\begin{proof} Indeed, \eqref{ScatProd.91} is just a restatement of the
  chain-closure condition. The intersection-closure property
  \eqref{ScatProd.88} implies the face-closure condition by
  applying it with $A_1=A_2.$ Conversely \eqref{ScatProd.88} must always
  hold for an fc-closed collection in the sense defined above since
  $H_{A_i,\mathfrak{b}_i}\in\mathcal{G}$ for $i=1,2$ implies $H_{A_1\cap
    A_2,\mathfrak{b}_i}\in\mathcal{G}$ for $i=1,2$ by chain-closure and
  then $H_{A_1\cap A_2,\mathfrak{b}_1\Cup\mathfrak{b}_2}\in\mathcal{G}$ by
  face-closure.
\end{proof}

\begin{remark}\label{24.5.2008.121} Thus Proposition~\ref{ScatProd.86}
  applies to an fc-closed collection of boundary diagonals. It also follows
  that the part of an fc-closed collection
  $\mathcal{G}\subset\mathcal{H}_{*,\mathcal{C}}$ which occurs before any
  given point in a chain-order is also fc-closed 
  and chain-ordered, since the elements the existence of which is required
  by \eqref{ScatProd.60} and \eqref{ScatProd.91} must occur earlier in the
  chain-order. 
\end{remark}

\section{Scattering configuration spaces}\label{Scatcon}

In this section we complete the definition of the scattering
configuration space $X^n_{\scat}.$  This is defined by blowing up the
boundary d-collection of boundary multi-diagonals in $X^n_{\bo}.$

\begin{definition}\label{24.5.2008.163} The n-fold scattering configuration
  space (or stretched product) of a compact manifold with boundary is
  defined to be
\begin{equation}
X_{\scat}^n=[X^n_{\bo};H_{*,\cB_{\bo}}]
\label{24.5.2008.164}\end{equation}
where the boundary diagonals are to be blown up in a chain-order.
\end{definition}

That this manifold exists and is independent of choice of the chain-order
chosen follows from the fact that Proposition~\ref{ScatProd.86} certainly
applies to the collection of all boundary diagonals when $\cC=\cB_{\bo}.$
Moreover the same argument applies to show the symmetry of the resulting object.

\begin{proposition}\label{24.5.2008.165} The permutation group $\Sigma _n$
  acts on $X^n_{\scat}$ as the lifts of the factor exchange diffeomorphisms
  of $X^n.$
\end{proposition}

\begin{proof} This just amounts to carrying out the blow up in
  \eqref{24.5.2008.164} in a different chain-order.
\end{proof}

This means that to construct all the maps $X^n_{\scat}\longrightarrow
X^m_{\scat},$ $m<n,$ covering the projections off various factors, it
suffices to consider the case $m=n-1$ with the last factor projected off
and then apply permutations and compose. This is discussed in detail in
\S\ref{Sect-Ssp} where the arguments depend on more complicated commutation
results which we proceed to discuss.

\section{Reordering blow-ups}\label{Reorder}

From Proposition~\ref{ScatProd.86} it follows that the blow up in
$[X^n;\mathcal{C}]$ of an fc-closed subcollection
$\mathcal{G}\subset\mathcal{H}_{*,\mathcal{C}},$ where
$\mathcal{C}\subset\mathcal{B}_{\bo}$ is closed under non-transversal
intersection, is iteratively defined with respect to any chain-order and
the final result is a manifold with corners which is independent of the
chain-order. Thus, under these conditions, $[X^n;\mathcal{C};\mathcal{G}]$
is well-defined. In this section we give three results which relate these
manifolds under blow-down.

\begin{remark}\label{ScatProd.166} In the blow up of an fc-closed
  collection $\cG$ of boundary diagonals the order is such that all lifts
  are closures of inverse images of complements with respect to the
  centre. It follows that the intersection properties of any two elements,
  meaning whether they are transversal, comparable or \ncnt, remain
  unchanged unless their (original and hence persisting) intersection is
  blown up in which case they become disjoint.
\end{remark}

\begin{proposition}\label{ScatProd.25} Let
$\mathcal{C}\subset\cB_{(2)}$ be closed under non-transversal
intersection and suppose $\mathcal{G}\subset\mathcal{H}_{*,\mathcal{C}}$ is an
fc-closed subcollection of boundary diagonals such that
\begin{equation}
H_{A,\mathfrak{a}}\in\mathcal{G}\Longrightarrow A\in\mathcal{C}
\label{ScatProd.27}\end{equation}
and consider a particular element $H_{C,\mathfrak{c}}\in\mathcal{G}$ with
the additional properties 
\begin{equation}
\begin{gathered}
\{H_{C,\mathfrak{b}}\in\mathcal{G};\mathfrak{b}\not=\mathfrak{c}\}
\text{ is closed under intersections}\Mand\\
\mathcal{C}\ni A\supsetneq C\Longrightarrow H_{A,\mathfrak{c}}\notin\mathcal{G}
\end{gathered}
\label{ScatProd.60}\end{equation}
then there is a blow-down map
\begin{equation}
[X^n;\mathcal{C};\mathcal{G}]
\longrightarrow
[X^n;\mathcal{C};\mathcal{G}'],\
\mathcal{G}'=\mathcal{G}\setminus\{H_{C,\mathfrak{c}}\}.
\label{9.1.2008.4}\end{equation}
\end{proposition}

\begin{proof} First observe that the collection $\mathcal{G}'$ is
  fc-closed. The chain-closure condition is just \eqref{ScatProd.91} and
  this holds for $\mathfrak{b}\not=\mathfrak{c}$ since it holds for
  $\mathcal{G}.$ For $\mathfrak{b}=\mathfrak{c}$ the only danger is that
  the element in question is $H_{C,\mathfrak{c}}.$ However, the second
  assumption in \eqref{ScatProd.60} shows that $C$ cannot be a boundary
  face of $A$ with $H_{A,\mathfrak{c}}\in\mathcal{G}'.$ Similarly if
  $H_{A_i,\mathfrak{b}_i}\in\mathcal{G}'$ for $i=1,$ $2,$ the element
  required in \eqref{ScatProd.88}, which exists in $\mathcal{G}$ by
  hypothesis, is certainly in $\mathcal{G}'$ unless $A_1\cap A_2=C$ and
  $\mathfrak{c}=\mathfrak{b}_1\Cup\mathfrak{b}_2.$ Thus $C$ must then be a
  boundary face of both $A_1$ and $A_2.$ It cannot be that
  $\mathfrak{b}_i=\mathfrak{c}$ for either $i=1$ or $2$ since this would
  mean $A_i=C$ by the second part of \eqref{ScatProd.60} and hence the
  corresponding boundary diagonal would not be in $\mathcal{G}'.$ Thus
  $\mathfrak{b}_i\neq\mathfrak{c}$ and then \eqref{ScatProd.91} implies
  that $H_{C,\mathfrak{b}_1}$ and $H_{C,\mathfrak{b}_2}\in\mathcal{G}.$
  Their intersection being $H_{C,\mathfrak{c}}$ then violates the first
  condition in \eqref{ScatProd.60} so \eqref{ScatProd.88} does hold for
  $\mathcal{G}'.$

Thus the right side of \eqref{9.1.2008.4} is indeed defined. The body of
the proof below is devoted to showing that if $H_{A,\mathfrak{a}}$ is the
last element in $\mathcal{G}'$ with respect to a chosen chain-order and
$\widetilde{\mathcal{G}}$ is $\mathcal{G}'$ with this removed then  
\begin{equation}
[X^n;\mathcal{C};\widetilde{\mathcal{G}};H_{C,\mathfrak{c}};H_{A,\mathfrak{a}}]
=[X^n;\mathcal{C};\widetilde{\mathcal{G}};H_{A,\mathfrak{a}};H_{C,\mathfrak{c}}]
\label{24.5.2008.135}\end{equation}
including of course showing that both are defined.

To see that this identity proves the Proposition, observe, following
Remark~\ref{24.5.2008.121}, that the $\mathcal{G}_j$ obtained from
$\mathcal{G}$ by dropping the last $j$ terms for $1\le j\le j',$ where
$H_{C,\mathfrak{c}}$ is $j'+1$ terms from the end, are fc-closed and
chain-ordered. Moreover, the conditions of the Proposition hold for all
these $j.$ Then iterating \eqref{24.5.2008.135} shows that all the
manifolds obtained by blowing up $H_{C,\mathfrak{c}}$ at some later point
are all canonically diffeomorphic and hence there is a blow-down map
\eqref{9.1.2008.4}.

Thus we are reduced to showing \eqref{24.5.2008.135} under the assumptions
of the Proposition. To do so we consider that the intersection properties of
$H_{C,\mathfrak{c}}$ and $H_{A,\mathfrak{a}}$ (which is the last element of
$\mathcal{G})$ in the manifold $M_1=[X^n;\mathcal{C}]$ and subsequently in
the manifold $M_2$ which is $M_1$ with all elements preceding $H_{C,\fc}$
blown up.

All boundary faces $B$ forming the boundary-hull of elements
$H_{B,\mathfrak{b}}\in\mathcal{G}$ are, by assumption in
\eqref{ScatProd.27}, in $\mathcal{C}$ and hence have already been blown up
in $M_1.$ In particular if the intersection $A\cap C$ is \ncnt, then $A\cap
C$ has been blown up, the lifts of $A$ and $C$ are disjoint and hence so
are $H_{C,\mathfrak{c}}$ and $H_{A,\mathfrak{a}}.$ On the other hand, if
$A\pitchfork C$ in $X^n$ then $\cap\mathfrak{a}$ and $\cap\mathfrak{c},$
which contain them, must also be transversal so in fact $\mathfrak{a}$ and
$\mathfrak{c}$ must meet transversally as subsets of $\cB_{(2)}.$ It
follows that $H_{A,\mathfrak{a}}$ and $H_{C,\mathfrak{c}}$ are transversal
in $M_1.$ By Remark~\ref{ScatProd.166} they remain transversal in $M_2.$

So we need consider the intersection properties of $H_{C,\fc}$ and a later
$H_{A,\fa}$ with $A$ and $C$ comparable. The chain-order condition on
blow-ups, in which $H_{A,\mathfrak{a}}$ comes after $H_{C,\mathfrak{c}}$
implies that $\mathfrak{a}>\mathfrak{c}$ since the second part of
\eqref{ScatProd.60} implies that $H_{C,\fc}$ is the last element
corresponding to the diagonal $D_{\fc}.$ It follows that $\fa\Cup\fc<\fc$
unless $\fc\Subset\fa.$ In the first case $H_{A\cap C,\fa\Cup\fc}$ precedes
$H_{C,\fc}$ and so has already been blown up, and hence by
Remark~\ref{ScatProd.166}, $H_{C,\fc}$ and $H_{A,\fa}$ are disjoint at some
point before $H_{C,\fc}$ in the order of $\cG$ and so remain so in $M_2.$

Summarizing we see that 
\begin{equation}
\begin{gathered}
\text{Either }A\Mand C\text{ not comparable }\Mor\fa\not\Subset\fc
\Longrightarrow\\
H_{C,\fc}\Mand
H_{A,\fa}\text{ are transversal (or disjoint) in }M_2.
\end{gathered}
\label{ScatProd.170}\end{equation}

Thus it remains to consider the cases in which $A$ and $C$ are comparable
and the diagonals $D_{\fa}$ and $D_{\fc}$ are also comparable.

Suppose first that $A\subset C$ with strict inclusion, before blow up. Then
$H_{A,\mathfrak{c}}$ has been blown up earlier and 
\begin{equation}
A\subsetneq C, \fc\Subset\fa\Longrightarrow H_{C,\fc}\Mand H_{A,\fa}\text{
  are disjoint in }M_2
\label{ScatProd.169}\end{equation}
by Remark~\ref{ScatProd.166}. Next
\begin{equation}
A=C,\ \fc\Subset\fa
\Longrightarrow H_{C,\mathfrak{c}}\subset H_{C,\mathfrak{a}}\Min M_2
\label{ScatProd.167}\end{equation}
since this is true in $M_1$ and not affected by subsequent blow
ups. Finally 
\begin{equation}
C\subsetneq A,\ \fc\Subset\fa
\Longrightarrow H_{C,\fc}\subset H_{A,\fa}\Min M_2\Mand
H_{C,\fa}\in\mathcal{G}.
\label{ScatProd.168}\end{equation}

Combining \eqref{ScatProd.170}, \eqref{ScatProd.169}, \eqref{ScatProd.167}
and \eqref{ScatProd.168} we see that in $M_2,$ i\@.e\@.~immediately before
$H_{C,\fc}$ is to be blown up in $\cG,$ all the subsequent lifted boundary
diagonals are transversal (including disjoint) or closed else $H_{C,\fc}$
is contained in $H_{A,\fc}$ with $H_{C,\fa}$ coming earlier if $C\subsetneq
A.$ From this \eqref{24.5.2008.135} follows.

This completes the proof of \eqref{24.5.2008.135} and hence of the Proposition.
\end{proof}

Next we consider a result which allows a boundary face to be blown down
without blowing down the boundary diagonals which are (or rather were)
contained in it provided the diagonals do not `involve' the given boundary
face.

\begin{proposition}\label{ScatProd.26}  Let
  $B\in\mathcal{C}\subset\cB_{(2)}$ be such that both $\mathcal{C}$ and 
  $\mathcal{C}'=\mathcal{C}\setminus\{B\}$ are closed under non-transversal
  intersections and suppose
  $\mathcal{G}\subset\mathcal{H}_{*,\mathcal{C}'}$ is an fc-closed
  subcollection of boundary diagonals such that in addition
\begin{equation}
\begin{gathered}
B\text{ comes last in some intersection-order of }\mathcal{C}\Mand\\
H_{A,\mathfrak{b}}\in\mathcal{G}\text{ implies }A\in\mathcal{C}
\end{gathered}
\label{ScatProd.28}\end{equation}
then there is a blow-down map 
\begin{equation}
[X^n;\mathcal{C};\mathcal{G}]\longrightarrow
[X^n;\mathcal{C}';\mathcal{G}].
\label{9.1.2008.5}\end{equation}
\end{proposition}

\begin{proof} By hypothesis both sides of \eqref{9.1.2008.5} are well
  defined and at least at a formal level differ by the blow-up of $B.$
  Furthermore, the hypotheses continue to hold if the `tail' of
  $\mathcal{G}$ is cut off at any point, with respect to a given
  chain-order. Thus we only need to show that under the given hypotheses,
  blowing up (the lift of) $B$ in $[X^n;\mathcal{C}';\mathcal{G}]$ gives
  the same manifold as blowing it up before the last element of
  $\mathcal{G}$ since then we can use induction over the number of elements
  in $\mathcal{G}$ to prove \eqref{9.1.2008.5} in the general case.

So, let the last element of $\mathcal{G}$ be $H_{A,\mathfrak{a}}.$ Thus
$\mathfrak{a}\subset\mathcal{C}'$ is a transversal collection of boundary
faces and by the second hypothesis, either $A=B$ or else
$A\in\mathcal{C}'.$ Now, if $B$ and $A$ are \ncnt, then $A\cap B$ must
already have been blown up and as a result $B$ and $H_{A,\mathfrak{a}}$ are
disjoint in $[X^n;\mathcal{C}']$ and this must remain true at the point of
interest, so commutation is possible. If $A$ and $B$ are transversal, then
so are $H_{A,\mathfrak{a}}$ and $B$ for any $\mathfrak{a}$ so the same
conclusion follows.

Thus we need only consider the case that $A$ and $B$ are
comparable. Suppose first that $B\subset A$ is a strict inclusion. Then, by
the chain condition on $\mathcal{G},$ $H_{B,\mathfrak{a}}\in\mathcal{G}$
must already have been blown up. However, initially, this is the
intersection of $H_{A ,\mathfrak{a}}$ and $B$ and by
Remark~\ref{ScatProd.166} remains so until it is blown up. Thus the
manifolds $H_{A ,\mathfrak{a}}$ and $B$ must be disjoint at the point of
interest and commutation is trivial.

Thus we may suppose conversely that $B\supset A,$ at first strictly. Then
$A\in\mathcal{C}'$ has been blown up, and in doing so it becomes
transversal to $B$ and this implies that $B$ is transversal to
$H_{A,\mathfrak{a}}$ so this remains true at the end of the blow up of
$\mathcal{G}$ and commutation of the blow up of $B$ and
$H_{A,\mathfrak{a}}$ is again possible.

Thus only the case $B=A$ remains. Then $A=B\subset\cap\mathfrak{a}$ and
indeed $H_{A,\mathfrak{a}}=H_{B,\mathfrak{a}}\subset B.$ By
Remark~\ref{ScatProd.166}, this must still be true at the point of interest
since no $H_{B,\mathfrak{c}}$ containing $B_{B,\mathfrak{a}}$ can have been
blown up (since $\mathcal{G}$ is chain-ordered), so commutation is again
possible.

This proves, inductively, that there is a blow-down map \eqref{9.1.2008.5}.
\end{proof}

The third result in this section corresponds to blowing down boundary
diagonals with boundary hull which has not been blown up.

\begin{proposition}\label{ScatProd.29} Suppose
  $\mathcal{C}\subset\cB_{(2)}$ and $B\in\mathcal{C}$ is such that
  both $\mathcal{C}$ and $\mathcal{C}'=\mathcal{C}\setminus\{B\}$ are
  closed under passage to boundary faces (in $X^n)$ and
  $\mathcal{G}\subset\mathcal{H}_{*,\mathcal{C}'}$ is fc-closed and
 such that 
\begin{equation}
H_{A,\mathfrak{b}}\in\mathcal{G}\Longrightarrow A\in\mathcal{C}
\label{ScatProd.30}\end{equation}
then there is an iterated blow-down map 
\begin{equation}
[X^n;\mathcal{C}';\mathcal{G}]\longrightarrow [X^n;\mathcal{C}';\mathcal{G}'],\
\mathcal{G}'=\{H_{A,\mathfrak{b}}\in\mathcal{G};A\not=B\}.
\label{ScatProd.31}\end{equation}
\end{proposition}

\begin{proof} Under the hypotheses of the Proposition it suffices to show
  that there is a blow-down map 
\begin{equation}
[X^n;\mathcal{C}';\mathcal{G}]\longrightarrow [X^n;\mathcal{C}';\mathcal{L}]
\label{ScatProd.32}\end{equation}
where $\mathcal{G}=\mathcal{L}\cup\{H_{B,\mathfrak{b}}\}$ with
$H_{B,\mathfrak{b}}\in\mathcal{G}$ where $\mathfrak{b}$ is of minimal
codimension for the existence of such an element. The general case then
follows by iteration. Thus we need to show that $H_{B,\mathfrak{b}}$ can be
blown down, and to do this it needs to be commuted through the boundary
diagonal faces $H_{A,\mathfrak{a}}$ which come after it. We can take the
chain-order of $\mathcal{G}$ so that $H_{B ,\mathfrak{b}}$ is followed only
by elements $H_{C,\mathfrak{b}}$ with $\codim C<\codim B$ or by
$H_{A,\mathfrak{a}}$ with $\mathfrak{a}$ of smaller codimension than
$\mathfrak{b}.$

Consider the intersection properties of $H_{B,\fb}$ and $H_{A,\fa}$ in
$M=[X^n;\mathcal{C}']$ with the boundary diagonals preceding $H_{B,\fb}$
also blown up.  Certainly neither $A$ nor $C$ can contain $B$ since then
$\mathcal{C}'$ would not contain the boundary faces of all its elements. It
follows that $H_{C\cap B,\mathfrak{b}}$ has already been blown up so
$H_{B,\mathfrak{b}}$ and $H_{C,\mathfrak{b}}$ are disjoint, by
Remark~\ref{ScatProd.166}. So, consider a later element
$H_{A,\mathfrak{a}}\in\mathcal{G}$ where
$\codim(\mathfrak{a})<\codim(\mathfrak{b}).$ As already noted, $A\supset B$
is not possible. Thus either $A$ and $B$ are not comparable or $A\subset B$
strictly. In either first case, $A\cap B\in\mathcal{C}'$ has been blown up
so $H_{B,\fb}$ and $H_{A,\fa}$ are disjoint. In the second,
$H_{A,\fa\Cup\fb}$ has been blown up and again these manifolds are disjoint
in $M.$

Thus all the boundary diagonals following $H_{B,\fb}$ are disjoint from it
and commutation is possible, proving \eqref{ScatProd.32} and hence the
Proposition.
\end{proof}

\section{Scattering stretched projections}\label{Sect-Ssp}

For any compact manifold $X$ with connected boundary the n-fold scattering
space is defined has been defined in \S\ref{Scatcon}.

\begin{theorem}\label{ScatProd.94} The stretched boundary projection off
  any $n-l$ factors lifts to a uniquely defined smooth map which is a
  b-fibration giving a commutative diagramme with the vertical blow-down maps
\begin{equation}
\xymatrix{X^n_{\scat}\ar^{\pi_{\scat}}[r]\ar[d]&X^l_{\scat}\ar[d]\\
X^n_{\bo}\ar^{\pi_{\bo}}[r]\ar[d]&X^l_{\bo}\ar[d]\\
X^n\ar^{\pi}[r]&X^l.}
\label{ScatProd.95}\end{equation}
\end{theorem}

\begin{proof} It suffices to prove this in case $l=n-1$ with the last
  factor projected off and then iterate; the lower part of the diagramme
  has already been constructed in Proposition~\ref{ScatProd.5}. We proceed
  to show there is an iterated blow-down map 
\begin{equation}
\beta :X^n_{\scat}\longrightarrow X^{n-1}_{\scat}\times X
\label{ScatProd.96}\end{equation}
starting from the corresponding map for the boundary stretched products 
\begin{equation}
\beta :X^n_{\bo}\longrightarrow X^{n-1}_{\bo}\times X.
\label{ScatProd.97}\end{equation}
The scattering n-fold stretched product is obtained from $X^n_{\bo}$ by
blowing up, in chain-order, all the $H_{A,\mathfrak{b}}.$ Each
$\mathfrak{b}\subset\cB_{(2)}$ either has $\cap\mathfrak{b}\subset
X^{n-1}\times\pa X,$ or not, and similarly either $A\subset X^{n-1}\times
\pa X$ or not. This leads to the decomposition
\begin{equation}
\cH_{*,*}=\cH_{\ver,\ver}\cup\cH_{\nver,\ver}\cup\cH_{\nver,\nver}
\label{ScatProd.98}\end{equation}
where the first part corresponds to those $H_{A,\mathfrak{b}}$ with
$\cap\mathfrak{b}\subset X^{n-1}\times\pa X$ (and hence also $A\subset
X^{n-1}\times\pa X,$ the second part to those with $A\subset
X^{n-1}\times\pa X$ but $\cap\mathfrak{b}\cap X^{n-1}\times(X\setminus\pa
X)\not=\emptyset$ and the third part being the remainder. Note that the
implied order here is very far from a chain-order.

Since $\cH_{\ver,\ver}$ is just the lift from $X^{n-1}_{\bo}$ of all the
boundary diagonals there 
\begin{equation}
[X^n;\mathcal{B}_{(2)}^{\ver};\cH_{\ver,\ver}]=X^n_{\scat}\times X
\label{24.5.2008.154}\end{equation}
which is the space we wish to map to. Thus we need to show 
that the elements of $\cH_{\nver,\nver}$ and $\cH_{\nver,\ver}$ can all be
blown down.

Proposition~\ref{ScatProd.25} applies directly to the elements of
$\cH_{\nver,\nver}.$ Thus we define successive subsets of the set of all
boundary diagonals by dropping the elements of $\cH_{\nver,\nver},$
starting with the last, in an overall chain-order. Then \eqref{ScatProd.27}
holds for each $\mathcal{G}$ so constructed as do the two conditions in
\eqref{ScatProd.60}. Thus, the elements of $\cH_{\nver,\nver}$ can indeed
be blown down in reverse order and we have shown the existence of an
iterated blow down map 
\begin{equation}
X^n_{\scat}\longrightarrow [X^n;\cB_{(2)};\cH_{\ver,\ver}\cup\cH_{\nver,\ver}].
\label{24.5.2008.155}\end{equation}

Now, we order the boundary faces of $X^n$ by taking first the elements of
$\cB_{(2)}^{\ver},$ products of the boundary faces of $X^{n-1}$ with $X,$ in
size-order, and then the remainder which form $\cB_{(2)}^{\nver}$ in a
size-order. Overall this is an intersection order so
$X^{n}_{\bo}=[X^{n-1}_{\bo}\times X;\mathcal{B}_{(2)}^{\ver}].$

Propositions~\ref{ScatProd.26} and \ref{ScatProd.29} can now be applied to
show the existence of successive blow-down maps
\begin{equation}
[X^{n-1}_{\bo}\times X;\mathcal{B};\cH_{\ver,\ver};\cH_{\nver,\ver}]\longrightarrow
[X^{n-1}_{\bo}\times X;\mathcal{B}^{\nver}(j);\cH_{\ver,\ver};\cH_{\nver,\ver}(j)]
\label{ScatProd.100}\end{equation}
where $\mathcal{B}^{\nver}(j)\subset\mathcal{B}^{\nver}$ is obtained by
removing the last element $j$ elements and $\cH_{\nver,\ver}(j)\subset
\cH_{\nver,\ver}$ is obtained by removing all those $H_{B,\ver}$
corresponding to the $B\in\mathcal{B}^{\nver}\setminus\mathcal{B}^{\nver}(j).$
This can be accomplished inductively, so it suffices to show that we can
pass from the space on the right in \eqref{ScatProd.100} for $j$ to the
same space for $j+1.$ Thus, the objective is to remove the last element of
$\cB^{\nver}_{(2)};$ Proposition~\ref{ScatProd.26} shows that this is
possible. Thus there is a blow down map 
\begin{equation}
[X^{n-1}_{\bo}\times X;\mathcal{B}^{\nver}(j);\cH_{\ver,\ver};\cH_{\nver,\ver}(j)]
\longrightarrow
[X^{n-1}_{\bo}\times X;\mathcal{B}^{\nver}(j+1);\cH_{\ver,\ver};\cH_{\nver,\ver}(j)]
\label{24.5.2008.156}\end{equation}
for each $j.$ Now, Proposition~\ref{ScatProd.29} can then be applied to `remove,'
i\@.e\@. blow down, all elements of
$\cH_{\nver,\ver}(j)\setminus\cH_{\nver,\ver}(j);$ these are of the form
$H_{B,\mathfrak{b}}$ where
$\cB_{(2)}^{\nver}(j+1)=\cB_{(2)}^{\nver}(j)\cup\{B\}.$ Thus, by alternating
between Propositions~\ref{ScatProd.26} and \ref{ScatProd.29} we have shown
the existence of an iterated blow-down map 
\begin{equation}
X^{n}_{\scat}\longrightarrow [X^{n-1}_{\bo}\times
  X;\cH_{\ver,\ver}]=X^{n-1}_{\scat}\times X
\label{ScatProd.101}\end{equation}
as desired. Composing with the projection shows the existence of a b-map
$\pi_{\scat}$ giving the commutative diagramme of smooth maps \eqref{ScatProd.95}.

By construction the stretched projection is a b-map. To show that it is a
b-submersion we proceed backwards through the construction above. Starting
from $X^{n-1}_{\scat}\times X$ all of the blow-ups of elements of
$\cH_{\nver,\ver}$ and $\cB^{\nver}_{(2)}$ are of boundary faces. Thus
these blow-ups are indeed b-submersions. However, the remaining blow-ups,
of the elements of $\cH_{\nver,\nver}$ are not of boundary faces and so
these are not b-submersions. Rather it is only the map back to
$X^{n-1}_{\scat},$ i\@.e\@. after the projection off the factor of $X$ that
is a b-submersion.
\end{proof}

\section{Vector spaces}\label{Vspaces}

As one indication of the claimed `universality' of these
scattering-stretched products alluded to in the Introduction we consider
here the case of a vector space, $V$ over $\bbR.$ As noted in the
Introduction the `model' asymptotic translation-invariance structure for a
general manifold with boundary is the radial compactification
$\overline{V}$ of $V$ to a ball. This has useful analytic properties, for
instance the space $S^0(V)$ of classical (1-step) symbols on $V$ of order
$0$ is identified with $\CI(\overline{V}).$ The question we consider here
is the extension of the difference map 
\begin{equation}
\delta :V\times V\ni(u,v)\longmapsto u-v\in\overline{V}
\label{24.5.2008.125}\end{equation}
and its higher variants 
\begin{equation}
\delta _{ij}:V^n\overset{\pi_{ij}}\longrightarrow V\times V\overset{\delta
}\longrightarrow \overline{V}.
\label{24.5.2008.126}\end{equation}

\begin{proposition}\label{24.5.2008.127} The difference map 
  \eqref{24.5.2008.125} extends to a b-fibration $\tilde\delta
  :(\overline{V})^2_{\scat}\longrightarrow \overline{V}$ and the higher
  maps \eqref{24.5.2008.126} similarly extend to b-fibrations 
\begin{equation}
\tilde{\delta }_{ij}:(\overline{V})^n_{\scat}\longrightarrow \overline{V}
\label{24.5.2008.128}\end{equation}
for all $i\not=j.$ 
\end{proposition}

\begin{proof} The second result follows from the first and the properties
  of $(\overline V)^n_{\scat}$ since $\tilde{\delta}_{ij}=\tilde\delta
  \circ\pi_{ij,\scat}$ where $\pi_{ij,\scat}$ is the scattering stretched
  projection corresponding to $\pi_{ij}.$ As the product of two b-fibrations
  it is also a b-fibration. The proof that $\delta$ extends to a
  b-fibration is elementary. First, the difference extends smoothly to
  either $\overline{V}\times V$ or $V\times\overline{V}$ so it suffices to
  consider a small neighbourhood of $(\pa \overline{V})^2$ in
  $(\overline{V})^2.$ This is of the form $[0,1)^2\times\bbS^{n-1}$ with
    the boundary variables being inverted radial variables, so the
    difference, written in terms of inverted polar coordinates in
    $\overline{V}$ is
\begin{equation}
((x,\omega ),(s,\theta))\longmapsto (|s\omega -x\theta|,\frac{\omega/x
    -\theta/s}{|\omega/x -\theta/s|}).
\label{24.5.2008.129}\end{equation}
The blow up to $(\overline{V})^{2}_{\bo}$ replaces $x,s$ by $T=x+s$ and
$y=\frac{x-s}{x+s}$ so the difference map becomes 
\begin{equation}
X=T|(1-y)\omega -(1+y)\theta|,\
\phi =\frac{(1-y)\omega -(1+y)\theta}{|(1-y)\omega -(1+y)\theta|}\in\bbS^{n-1}.
\label{24.5.2008.130}\end{equation}
This is smooth away from $T=0,$ $\omega =\theta$ and the scattering blow
up, precisely of this set, resolves the singularity.
 \end{proof}

As an example consider the following result on `multilinear convolution'
examining integrals of the form 
\begin{equation}
\int_{V^n}a_{0,1}(z_0-z_1)\cdots a_{0,n}(z_0-z_n)a_{1,2}(z_1-z_2)\cdots
  a_{n-1,n}(z_{n-1}-z_n)dz_1\ldots dz_n.
\label{24.5.2008.133}\end{equation}

\begin{corollary}\label{24.5.2008.131} If $a_{i,j}\in\rho
  ^{n+1}\CI(\overline{V})$ are classical symbols of order $-n-1$ for all
  $0\le i<j\le n$ and  
\begin{equation}
A=\prod_{0\le i<j\le n}\tilde \delta_{ij}^*a_{i,j}\in\CI((\overline
V)^{n+1}_{\scat}) 
\label{24.5.2008.132}\end{equation}
then $A$ pushes forward to a classical symbol, with logs, on $\overline{V}.$ 
\end{corollary}

\providecommand{\bysame}{\leavevmode\hbox to3em{\hrulefill}\thinspace}
\providecommand{\MR}{\relax\ifhmode\unskip\space\fi MR }
\providecommand{\MRhref}[2]{%
  \href{http://www.ams.org/mathscinet-getitem?mr=#1}{#2}
}
\providecommand{\href}[2]{#2}


\end{document}